\documentclass[pdflatex,sn-basic]{sn-jnl}

\usepackage{graphicx}%
\usepackage{multirow}%
\usepackage{amsmath,amssymb,amsfonts}%
\usepackage{amsthm}%
\usepackage{mathrsfs}%
\usepackage[title]{appendix}%
\usepackage{xcolor}%
\usepackage{textcomp}%
\usepackage{manyfoot}%
\usepackage{booktabs}%
\usepackage{algorithm}%
\usepackage{algorithmicx}%
\usepackage{algpseudocode}%
\usepackage{listings}%

\renewcommand{\proofname}{\textbf{Proof}}

\newtheorem{theorem}{Theorem}[section]
\newtheorem{lemma}[theorem]{Lemma}      
\newtheorem{corollary}[theorem]{Corollary}

\newtheorem{example}[theorem]{Example}
\newtheorem{remark}[theorem]{Remark}
\begin{document}

\title[Dual orthogonal tripotent matrices]{Dual orthogonal tripotent matrices}


\author[1]{\fnm{Tan} \sur{Mei}}\email{meitan{\url_}0013@163.com}

\author*[1]{\fnm{Kezheng} \sur{Zuo}}\email{xiangzuo28@163.com}

\author[2]{\fnm{Yang} \sur{Si}}\email{christinasi1113@gmail.com}

\affil [1]{Department of Mathematics, School of Mathematics and Statistics, Hubei Normal University, Huangshi 435002, China}
\affil [2]{Department of Applied Physics and Applied Mathematics, Columbia University, New York, NY 10027, United States of America}


\abstract{In this paper, we study dual orthogonal tripotent matrices, defined by \unboldmath$\hat{A}^3=\hat{A} = \hat{A}^*$, and examine their fundamental algebraic properties.  
 Additionally, we establish several characterizations of this class of matrices using matrix averages involving \unboldmath$\hat{A}$, $\hat{A}_e$, $\hat{A}^*$, $\hat{A}\sp{\scriptscriptstyle N}$, 
as well as integer powers of products such as  \unboldmath$\hat{A}\hat{A}^*$ and $\hat{A}^*\hat{A}$.
These results enrich the theory of dual generalized matrix classes and reveal new perspectives on related dual quaternion matrices.}

\keywords{ New dual Moore-Penrose inverse, Dual orthogonal tripotent matrices, Dual orthogonal idempotent matrices, Dual Hartwig-Spindelb\"ock decomposition}


\pacs[MSC Classification]{15A09, 15A66}

\maketitle

\section{Introduction}

Recently, Mei et al. (\citeyear{Mei2}) investigated the orthogonal tripotent matrices, which are defined as complex square matrices  $A$ satisfying $A^3 = A = A^*$.
These matrices naturally generalize orthogonal idempotent matrices, i.e., $A^2=A=A^*$, which have been widely applied in diverse fields such as least squares problems in statistics, control theory, operator theory, and matrix equations (Conway \citeyear{conway}; To{\v{s}}i{\'c} and Cvetkovi{\'c}-Ili{\'c} \citeyear{To}; Pes and Rodriguez \citeyear{pes}).
 Given the theoretical significance and practical utility of orthogonal idempotent matrices, it is both natural and worthwhile to extend the concept of orthogonal tripotent matrices to the broader setting of dual quaternion matrices.
This leads to the introduction of dual orthogonal tripotent matrices, namely, dual quaternion matrices $\hat{A}$ satisfying $\hat{A}^3 = \hat{A} = \hat{A}^*$.

Dual quaternions, introduced by Clifford (\citeyear{Clifford}), offer an elegant and efficient mathematical framework for representing rigid body motions in three-dimensional space. 
With their distinctive algebraic structure, dual quaternion matrices have become powerful tools in robotics and motion control (Chen et al. \citeyear{chen}; Gu et al. \citeyear{Gu}; Giribet et al. \citeyear{Giribet}; Xiao et al. \citeyear{xiao}), kinematic analysis and synthesis of spatial mechanisms (Ding \citeyear{Ding}; Udwadia et al. \citeyear{Udwadia1}; Yang and Wang \citeyear{Yang}), among other fields.

 Under the Hermitian condition $\hat{A} = \hat{A}^*$, it follows from Qi and Luo (\citeyear{Qi1})  that $\hat{A}$ admits a unitary diagonalization over the dual numbers. 
This diagonalizability significantly simplifies the analysis of dual orthogonal tripotent matrices by reducing complex matrix operations to the manipulation of diagonal components. 
Building on this diagonalization property and considering the tripotent condition $\hat{A}=\hat{A}^3$, which generalizes $k$-idempotency over the complex field (Baksalary and Trenkler \citeyear{Baksalary2}), 
 it can be shown that the eigenvalues of dual orthogonal tripotent matrices are restricted to the real numbers  $1$, $-1$ and $0$.
This restriction provides a foundation for further exploring the algebraic and spectral properties of these matrices and facilitates their potential applications.

Motivated by these findings, the study of dual orthogonal tripotent matrices carries significant theoretical value and broad application prospects. Accordingly, this paper focuses on exploring their properties and characteristics.

The remaining content of this paper is organized as folleows. 
In Section \ref{222},  we review two types of dual generalized inverses and dual matrix classes, and present the necessary lemmas and theorems for the subsequent analysis.
In Section \ref{333}, we establish fundamental properties of dual orthogonal tripotent matrices and give several necessary and sufficient conditions for such matrices.  
Furthermore,  we identify relationships between $\hat{A}$, $\hat{A}^*$, $\hat{A}\sp{\scriptscriptstyle N}$, $\hat{A}_e$ and dual orthogonal tripotent matrices.
In Section \ref{444}, we characterize dual orthogonal tripotent matrices by means of dual quaternion matrices, dual generalized inverses, mean value type equations, linear equations, as well as integer powers of dual matrices.
Finally, concluding remarks are given in Section \ref{555}.

\section{Preliminaries}\label{222}

Throughout this paper the following notations and definitions are used.
 Let \( \mathbb{Z}\), \( \mathbb{R}\), \( \mathbb{D}\), \( \mathbb{Q}\) and \( \mathbb{Q}^{m \times n} \) be the sets of integers, real numbers, dual numbers, quaternions and \( m \times n \) quaternion matrices, respectively. 

Dual number \(\hat{a} = a_s + \epsilon a_d\) is defined by two real components, the standard part \(a_s\) as and the infinitesimal part \(a_d\),  along with the infinitesimal unit $\epsilon$, satisfying
$\epsilon \neq 0$, $\epsilon^2 = 0$, $ 0\epsilon = \epsilon 0 = 0$ and $1\epsilon = \epsilon 1 = \epsilon$.
The dual number \( \hat{a} \) is appreciable if \( a_s \neq 0 \),  and \( \hat{a} \) is invertible if and only if it is appreciable, with the inverse given by $ \hat{a}^{-1}=a_s^{-1}-\epsilon a_s^{-2}a_d$.
For any positive integer $p \in \mathbb{Z}$, we have
$\hat{a}^p = a_s^p + \epsilon p a_s^{p-1}a_d.$
For any $\hat{a}=a_s+\epsilon a_d \in \mathbb{D}$ and $\hat{b}=b_s+\epsilon b_d \in \mathbb{D}$, we have $\hat{a}=\hat{b}$ if and only if $a_s=b_s$ and $a_d=b_d$.
Furthermore, $\hat{a}>\hat{b}$ if and only if either $a_s>b_s$, or $a_s=b_s$ and $a_d>b_d$. 
Consequently, a dual number $\hat{a}$ with $a_s>0$ is called positive appreciable, while one with $a_s=0$ and $a_d>0$ is termed positive infinitesimal.

A quaternion $m\in \mathbb{Q}$ takes the form $m = m_0 + m_1\mathbf{i} + m_2\mathbf{j} + m_3\mathbf{k}$, where $m_0, m_1, m_2,m_3 \in \mathbb{R}$, , and the imaginary units $\mathbf{i}, \mathbf{j}$ and $\mathbf{k}$ satisfy the following identities: $\mathbf{i}^2 = \mathbf{j}^2 = \mathbf{k}^2 = \mathbf{ijk} = -1$, $\mathbf{ij} = -\mathbf{ji} = \mathbf{k}$, $\mathbf{jk} = -\mathbf{kj} = \mathbf{i}$, $\mathbf{ki} = -\mathbf{ik} = \mathbf{j}$. 
The conjugate of $m$ is $m^*= m_0 - m_1\mathbf{i} - m_2\mathbf{j} - m_3\mathbf{k}$.

Recently, Qi et al. (\citeyear{Qi2}) introduced that the set of dual quaternions is denoted by $\mathbb{DQ}$.  
A dual quaternion number \(\hat{a} \in \mathbb{DQ}\) is expressed as
$\hat{a}=a_s+\epsilon a_d,$
where \(a_s,a_d \in \mathbb{Q}\) represent the standard and infinitesimal parts, respectively. 
In addition, for any $\hat{a} \in \mathbb{D}$ and $\hat{b} \in \mathbb{DQ}$, we have  $\hat{a}\hat{b}=\hat{b}\hat{a}$.
The  conjugate of \(\hat{a} \in \mathbb{DQ}\) is defined as \(\hat{a}^*=a_s^*+\epsilon a_d^*\). 

Since each entry of a dual quaternion matrix is a dual quaternion, any matrix \( \hat{A} \in \mathbb{DQ}^{m \times n} \) can be written as
$$\hat{A} = A_s + \epsilon A_d,$$
where \( A_{s}, A_{d} \in \mathbb{Q}^{m \times n} \) are referred to as the standard and infinitesimal part of \( \hat{A} \), respectively.
For $A =(a_{ij}) \in \mathbb{Q}^{m \times n}$, the conjugate transpose is defined as $A^*=(a_{ji}^*)\in \mathbb{Q}^{m \times n}$.
Accordingly, for $\hat{A}\in \mathbb{DQ}^{m \times n}$,  $\hat{A}^* = A_s^* + \epsilon A_d^*$.
For any $\hat{A},\hat{B} \in \mathbb{DQ}^{m \times n}$,  $\hat{A}=\hat{B}$ if and only if  $A_s=B_s$ and $A_d=B_d$.
Additionally,  $(\hat{A}\hat{B})^*=\hat{B}^*\hat{A}^*$.

The $n \times n$ identity matrix is denoted by $I_n$. 
The set  $\mathbb{DQ}_n^{H}$ consists of $n \times n$ dual quaternion Hermitian matrices, where $\hat{A}\in\mathbb{DQ}_n^{H}$ satisfies $\hat{A}=\hat{A}^*$. 
The set $\mathbb{DQ}_n^{U}$ contains $n \times n$ dual quaternion unitary matrices, where $\hat{A}\in\mathbb{DQ}_n^{U}$ satisfies $\hat{A}\hat{A}^*=I_n$.
Finally, $\operatorname{rank}(A_s)$  denotes the rank of $A_s$, while $\operatorname{ARank}(\hat{A})$ denotes the appreciable rank of $\hat{A}$, given by $\operatorname{ARank}(\hat{A})=\operatorname{rank}(A_s)$.

Denote by $\mathbb{DQ}^{OP}$, $\mathbb{DQ}^{3\text{-}OP}$ the sets of dual orthogonal idempotent matrices, dual orthogonal tripotent matrices, respectively; that is,
\begin{align*}
&\mathbb{DQ}_n^{OP} = \{A\in\mathbb{DQ}^{n\times n}:\hat{A}^2 =  \hat{A}=\hat{A}^* \},\\
&\mathbb{DQ}_n^{3\text{-}OP}= \{A\in\mathbb{DQ}^{n\times n}:\hat{A}^3 =  \hat{A}=\hat{A}^*\}.
\end{align*}
In addition,  noted that $\mathbb{DQ}_n^{OP}\subseteq \mathbb{DQ}_n^{3\text{-}OP}$.

For dual quaternion matrices  $\hat{A} \in \mathbb{DQ}^{m \times n}$ and $\hat{X} \in \mathbb{DQ}^{n \times m}$, if 
\begin{align}\label{FF}
\hat{A}\hat{X}\hat{A} = \hat{A}, \quad  \hat{X}\hat{A}\hat{X} = \hat{X}, \quad  \left( \hat{A}\hat{X} \right)^* = \hat{A}\hat{X}, \quad  \left( \hat{X}\hat{A} \right)^* = \hat{X}\hat{A},
\end{align}
then  $\hat{X}$ is called the dual Moore-Penrose generalized inverse (DMPGI for short) of $\hat{A}$ (Udwadia et al. \citeyear{Udwadia1}), denoted by $\hat{A}^{\dagger}$.
However, the DMPGI of a dual quaternion matrix does not always exist in general (Wang \citeyear{Wang2}). Therefore, Li and Wang (\citeyear{Li1}) and Cui and Qi (\citeyear{Cui1}) introduced the new dual Moore-Penrose inverse (NDMPI for short).
For $\hat{A} \in \mathbb{DQ}^{m \times n}$, the NDMPI of \( \hat{A} \),  denote by \( \hat{A}\sp{\scriptscriptstyle N} \) (Cui and Qi \citeyear{Cui1}),  is defined as the unique dual quaternion matrix $\hat{X}$ satisfying
\begin{align}\label{FF1}
\hat{A} \hat{X} \hat{A} = \hat{A}_e, \quad \hat{X}\hat{A}\hat{X} = \hat{X}, \quad \left( \hat{A}\hat{X} \right)^* = \hat{A}\hat{X}, \quad \left( \hat{X}\hat{A} \right)^* = \hat{X}\hat{A},
\end{align}
where \( \hat{A}_e \) is the  essential part of \( \hat{A} \), i.e., the best rank-$r$ approximation of \( \hat{A} \) (Qi and Luo \citeyear{Qi1}).
 
Additionally, Qi and Luo (\citeyear{Qi1}) have established a unitary decomposition for dual quaternion Hermitian matrices.
\begin{lemma}\label{L01}\emph{(Qi and Luo \citeyear{Qi1})}
  Let $ \hat{A} \in \mathbb{DQ}_n^H $ with $r=\operatorname{ARank}(\hat{A})$. Then there exists \( U \in \mathbb{DQ}_n^{U} \)  such that 
$$\hat{U}^* \hat{A} \hat{U}= \Sigma
=\mathrm{diag} (\lambda_1, \cdots, \lambda_r, 0, \cdots, 0) ,$$
where $\lambda_i \in \mathbb{D}$ and $\lambda_i \neq 0$ for $i=1,2,\cdots,r$ are the  right eigenvalues of \( \hat{A} \).
\end{lemma}

 To facilitate the analysis of dual quaternion matrices, the dual Hartwig-Spindelb\"ock decomposition is introduced.
\begin{lemma}\emph{(Mei et al. \citeyear{mei1})}
 Let $\hat{A} \in \mathbb{DQ}^{n \times n}$ with $r=\operatorname{ARank}(\hat{A})$ and $t=\operatorname{rank}(\hat{A})$.
Then there exists $\hat{U} \in \mathbb{DQ}_n^{U}$ such that
\begin{align}
 \hat{A} &=\hat{U}\begin{bmatrix}
\Sigma_1 K & \Sigma_1 L\\
\Sigma_2 M & \Sigma_2 N
\end{bmatrix}\hat{U}^* \label{F1} \\
&= \hat{U} \left(\begin{bmatrix}
\Sigma_{\mathrm{1s}} K_{1} & \Sigma_{\mathrm{1s}} L_{1} \\ 0 & 0
\end{bmatrix} +\epsilon \begin{bmatrix}
\Sigma_{\mathrm{1d}} K_{1}+\Sigma_{\mathrm{1s}} K_{2} & \Sigma_{\mathrm{1d}} L_{1}+\Sigma_{\mathrm{1s}} L_{2} \\
\Sigma_{\mathrm{2d}} M_{1} & \Sigma_{\mathrm{2d}}  N_{1}
\end{bmatrix}\right)\hat{U}^* \label{ab},
\end{align}
with
  \begin{align*}
  &\Sigma_1 = \Sigma_{\mathrm{1s}} +\epsilon \Sigma_{\mathrm{1d}} = \operatorname{diag}(\mu_1, \ldots, \mu_r)\ \  \text{and} \ \
  \Sigma_2 = \epsilon \Sigma_{\mathrm{2d}}= \operatorname{diag}(\mu_{r+1}, \ldots, \mu_{t}, 0, \ldots, 0),
  \end{align*}
where $\mu_1, \cdots, \mu_r$ are positive appreciable dual numbers, $\mu_{r+1},\cdots,\mu_t$ are positive infinitesimal dual numbers.
And
 \begin{align*}
    & K=K_1+\epsilon K_2 \in \mathbb{DQ}^{r \times r}, \ \quad \quad \ L=L_1+\epsilon L_2 \in \mathbb{DQ}^{r \times (n-r)}, \\
     &M=M_1+\epsilon M_2\in \mathbb{DQ}^{(n-r)\times r }, \ \  N=N_1+\epsilon N_2\in \mathbb{DQ}^{(n-r)\times (n-r) }
 \end{align*}  satisfy
\begin{align}\label{SS}
KK^*+LL^*=I_r, \quad KM^*+LN^*=0, \quad MM^*+NN^*=I_{n-r},
\end{align}
 which  also  give
\begin{align}
 &K_1 K_1^*+L_1L_1^*=I_r \quad \ \ \quad \text{and} \quad K_1 K_2^* +K_2 K_1^*+L_1 L_2^*+L_2 L_1^*=0,\label{S1} \\
&M_1 M_1^*+N_1N_1^*=I_{n-r} \quad \text{and} \quad K_1 M_1^* +L_1 N_1^*=0.\label{S2}
\end{align}
\end{lemma}
Another essential fact is that the dual quaternion matrices $K$, $L$, $M$  and  $N$ in (\ref{SS}) satisfy
\begin{align}\label{ZXC}
  L=0 \quad \Rightarrow \quad KK^*=K^*K=I_r, \ \ M=0 \ \ \text{and} \ \ NN^*=N^*N=I_{n-r}. 
\end{align}
Meanwhile, from (\ref{S1}) and (\ref{S2}), we obtain the following fact concerning the quaternion matrices $K_1$, $K_2$, $L_1$, $M_1$  and  $N_1$:
\begin{align}
  L_1=0 \ \ \Rightarrow \ \ &K_1K_1^*=K_1^*K_1=I_r \ \ \ \ \ \text{and} \ \ K_1K_2^*+K_2K_1^*=0, \label{ZXC1}\\  
  & N_1N_1^*=N_1^*N_1=I_{n-r} \ \ \text{and} \ \ M_1=0 . \label{ZXC2}
\end{align}

\begin{remark}\label{RR}
    Let $\hat{A} \in \mathbb{DQ}^{n \times n}$ satisfy (\ref{F1}). Then
\begin{align}
\hat{A}_e  
&=\hat{U}\begin{bmatrix}
\Sigma_1 K & \Sigma_1 L\\
0 & 0
\end{bmatrix}\hat{U}^*  \label{Ee} \\
&= \hat{U} \left(\begin{bmatrix}
\Sigma_{\mathrm{1s}} K_{1} & \Sigma_{\mathrm{1s}} L_{1} \\ 0 & 0
\end{bmatrix} +\epsilon \begin{bmatrix}
\Sigma_{\mathrm{1d}} K_{1}+\Sigma_{\mathrm{1s}} K_{2} & \Sigma_{\mathrm{1d}} L_{1}+\Sigma_{\mathrm{1s}} L_{2} \\
 0 & 0
\end{bmatrix} \right)\hat{U}^*. \label{abbe}
\end{align}
\end{remark}

Also, the NDMPI can be represented by the dual Hartwig-Spindelb\"ock decomposition.

\begin{lemma}\label{LN}\emph{(Mei et al. \citeyear{mei1})}
  Let $\hat{A} \in \mathbb{DQ}^{n \times n}$ with $r=\operatorname{ARank}(\hat{A})$. Then there exists $\hat{U} \in \mathbb{DQ}_n^{U}$ such that
\begin{align}
\hat{A}\sp{\scriptscriptstyle N} 
&=\hat{U}\begin{bmatrix}
  K^* \Sigma_1^{-1}  &  0 \\ L^* \Sigma_1^{-1} & 0
\end{bmatrix}\hat{U}^* \label{FQ} \\
&=  \hat{U}\left(\begin{bmatrix}
K_{1}^* \Sigma_{\mathrm{1s}}^{-1}  &  0 \\
L_{1}^* \Sigma_{\mathrm{1s}}^{-1} & 0
\end{bmatrix} + \epsilon \begin{bmatrix}
K_{2}^*\Sigma_{\mathrm{1s}}^{-1}-K_{1}^* \Sigma_{\mathrm{1s}}^{-2}\Sigma_{\mathrm{1d}}   & 0  \\
L_{2}^*\Sigma_{\mathrm{1s}}^{-1}-L_{1}^* \Sigma_{\mathrm{1s}}^{-2}\Sigma_{\mathrm{1d}}   & 0
\end{bmatrix} \right)\hat{U}^*, \label{NN}
\end{align} 
where $K=K_1+\epsilon K_2 \in \mathbb{DQ}^{r \times r}$, $L=L_1+\epsilon L_2 \in \mathbb{DQ}^{r \times (n-r)}$ satisfy 
$KK^* + LL^* = I_r$, 
and $\Sigma_1 = \Sigma_{\mathrm{1s}} + \epsilon \Sigma_{\mathrm{1d}} = \operatorname{diag}(\mu_1, \ldots, \mu_r)$ with $\mu_1 \geq \cdots \geq \mu_r$ being positive appreciable dual numbers.
Furthermore,  $\hat{A}^{\dagger}$ exists whenever $\Sigma_2=0$,  in which case $\hat{A}^{\dagger}=\hat{A}\sp{\scriptscriptstyle N}$ holds. 
\end{lemma}
\vspace{5pt}
Mei et al. (\citeyear{Mei2}) established a fundamental lemma on orthogonal tripotent matrices. For the study of dual orthogonal tripotent matrices, an analogous result for dual quaternions is needed and is provided below. 
It is noteworthy that this result is not merely a direct corollary of the previously mentioned lemma.

 \begin{theorem}\label{ll}
  Let $K \in \mathbb{DQ}^{r \times r}$ and $\Sigma_1 = \text{diag}(\sigma_1, \sigma_2, \ldots, \sigma_r)$ be a diagonal matrix such that each $ \sigma_{i} $ is a positive appreciable dual number for $i = 1,\ldots, r$. 
  Then $\Sigma_1 K = K \Sigma_1$ if and only if one of the following conditions holds:\\
\noindent $(a)$ $\Sigma_1^p K = K \Sigma_1^p$ for $p \in \mathbb{Z}\setminus\{0\}$;\\
\noindent $(b)$ $(\Sigma_1^p + \Sigma_1^q) K = K (\Sigma_1^p + \Sigma_1^q)$ for $p,q \in \mathbb{Z}$ with $pq \geq 0$ and $(p,q) \neq (0,0)$.
\end{theorem}
\begin{proof}
Necessity is evident; it remains to prove sufficiency. Let 
$$K = (k_{ij})_{r \times r} \ \text{with} \ k_{ij}=(k_{ij})_s +\epsilon (k_{ij})_d \in \mathbb{DQ}, \quad  \sigma_i=\sigma_{i,s} +\epsilon \sigma_{i,d} \in \mathbb{D}, $$ 
where $(k_{ij})_s, (k_{ij})_d \in \mathbb{Q}$ and $\sigma_{i,s}, \sigma_{i,d} \in \mathbb{R}$ for \(i,j=1,\dots,r\).
To show that $\Sigma_1 K = K\Sigma_1$, it suffices to verify
\begin{equation}\label{1}
\sigma_i k_{ij} = k_{ij} \sigma_j \quad (i,j = 1,\ldots, r). 
\end{equation}
\noindent $(a):$ Since $\Sigma_1^p K = K\Sigma_1^p$ holds if and only if $\Sigma_1^{-p} K = K\Sigma_1^{-p}$, we may assume $p > 0$.  \\

If $k_{ij} \neq 0$, then from $\Sigma_1^p K = K\Sigma_1^p$ it follows that
\begin{align}\label{7}
\sigma_i^p k_{ij} = k_{ij}\sigma_j^p \quad \Rightarrow \quad (\sigma_i^p - \sigma_j^p)k_{ij}=0.
\end{align}
\noindent For $p>0$, we have
\begin{align}
\sigma_i^p - \sigma_j^p &= \sigma_{i,s}^p - \sigma_{j,s}^p +\epsilon p(\sigma_{i,s}^{p-1}\sigma_{i,d}-\sigma_{j,s}^{p-1}\sigma_{j,d}) \notag \\ 
&= (\sigma_{i,s} - \sigma_{j,s})\left(\sum_{k=0}^{p-1}\sigma_{i,s}^{p-1-k}\sigma_{j,s}^{k}\right) +\epsilon p(\sigma_{i,s}^{p-1}\sigma_{i,d}-\sigma_{j,s}^{p-1}\sigma_{j,d}), \label{9}
\end{align}

If $(k_{ij})_s \neq 0$, (\ref{7})  implies $\sigma_i^p - \sigma_j^p=0$.
The positivity of $\sigma_{i,s},\sigma_{j,s}$ in (\ref{9}) then yields  $\sigma_{i,s} = \sigma_{j,s}$, which further leads to $\sigma_{i,d} = \sigma_{j,d}$. 
 Consequently, $\sigma_i = \sigma_j$, satisfying (\ref{1}).

If $(k_{ij})_s = 0$ but $k_{ij} \neq 0$, i.e., $k_{ij}=d\epsilon$, where $d \in \mathbb{Q}$ and $d \neq 0$,
(\ref{7}) reduces to
$$(\sigma_i^p - \sigma_j^p)k_{ij}=(\sigma_{i,s} - \sigma_{j,s})\left(\sum_{k=0}^{p-1}\sigma_{i,s}^{p-1-k}\sigma_{j,s}^{k}\right)d\epsilon=0.$$
The positivity of $\sigma_{i,s},\sigma_{j,s}$ implies $\sigma_{i,s} = \sigma_{j,s}$,
leading to 
\begin{align}\label{99}
(\sigma_{i,s}+\sigma_{i,d}\epsilon)d\epsilon = d\epsilon(\sigma_{j,s}+\sigma_{j,d}\epsilon).
\end{align}
As a result,  $\sigma_i k_{ij} = k_{ij} \sigma_j$, and (\ref{1}) is satisfied.

\noindent $(b):$ If $k_{ij} \neq 0$, then  $(\Sigma_1^p + \Sigma_1^q) K = K (\Sigma_1^p + \Sigma_1^q)$ implies that
\begin{equation}\label{2}
(\sigma_i^p+\sigma_i^q)k_{ij} = k_{ij}(\sigma_j^p+\sigma_j^q)
\quad \Rightarrow \quad (\sigma_i^p - \sigma_j^P + \sigma_i^q - \sigma_j^q)k_{ij} = 0. 
\end{equation}
\noindent For $p,q>0$, we have
\begin{align}
\sigma_i^p - \sigma_j^p + \sigma_i^q - \sigma_j^q &= (\sigma_{i,s} - \sigma_{j,s})\left[ \left(\sum_{k=0}^{p-1}\sigma_{i,s}^{p-1-k}\sigma_{j,s}^{k}\right)+ \left(\sum_{k=0}^{q-1}\sigma_{i,s}^{q-1-k}\sigma_{j,s}^{k}\right)\right] \notag \\
&+ \left[p(\sigma_{i,s}^{p-1}\sigma_{i,d}-\sigma_{j,s}^{p-1}\sigma_{j,d}) + q(\sigma_{i,s}^{q-1}\sigma_{i,d}-\sigma_{j,s}^{q-1}\sigma_{j,d})\right]\epsilon. \label{90}
\end{align}

If $(k_{ij})_s \neq 0$, (\ref{2}) yields $\sigma_i^p - \sigma_j^p + \sigma_i^q - \sigma_j^q=0$.
The positivity of $\sigma_{i,s},\sigma_{j,s}$ in (\ref{90}) implies $\sigma_{i,s} = \sigma_{j,s}$, and then 
$$ 
p\sigma_{i,s}^{p-1}(\sigma_{i,d} - \sigma_{j,d}) + q\sigma_{i,s}^{q-1}(\sigma_{i,d} - \sigma_{j,d}) = 0.  
$$
Therefore, it follows that
$$
(p\sigma_{i,s}^{p-1}+q\sigma_{i,s}^{q-1})(\sigma_{i,d} - \sigma_{j,d}) = 0 \ \ \Rightarrow \ \ \sigma_{i,d} = \sigma_{j,d}.
$$
Hence, $\sigma_i = \sigma_j$, ensuring (\ref{1}). Similarly for $p,q< 0$, (\ref{1}) holds.

If $(k_{ij})_s = 0$ but $k_{ij} \neq 0$, i.e., $k_{ij}=d\epsilon$, where  $d \in \mathbb{Q}$ and $d \neq 0$. 
For $p,q> 0$, it follows from (\ref{2}) and (\ref{90}) that
$$ 
(\sigma_i^p - \sigma_j^p + \sigma_i^q - \sigma_j^q)k_{ij}=(\sigma_{i,s} - \sigma_{j,s})\left[ \left(\sum_{k=0}^{p-1}\sigma_{i,s}^{p-1-k}\sigma_{j,s}^{k}\right)+ \left(\sum_{k=0}^{q-1}\sigma_{i,s}^{q-1-k}\sigma_{j,s}^{k}\right)\right]d\epsilon=0.
$$
Given the positivity of $\sigma_{i,s},\sigma_{j,s}$, we have $\sigma_{i,s} = \sigma_{j,s}$. A similar result holds for $p,q< 0$.
Thus, (\ref{99}) holds, and hence (\ref{1}) is valid. 

The case where $k_{ij}=0$ is trivial. The cases with $p=0$ or $q=0$ reduce to case $(a)$.
Therefore, (\ref{1}) holds for all $p,q\in\mathbb{Z}$ with $pq\geq 0$ and $(p,q)\neq(0,0)$. 
\end{proof}

Analogously to Theorem~\ref{ll}, we can also derive the result for quaternion matrices, but omit its proof here for brevity.
\begin{corollary}\label{C2.7}
  Let $K_1 \in \mathbb{Q}^{r \times r}$ and $\Sigma_{\mathrm{1s}} = \text{diag}(a_1, a_2, \ldots, a_r)$ be a diagonal matrix such that each $ a_{i} $ is a positive real number for $i = 1,\ldots, r$. 
  Then $\Sigma_{\mathrm{1s}} K_1 = K_1 \Sigma_{\mathrm{1s}}$ if and only if one of the following conditions holds:\\
\noindent $(a)$ $\Sigma_{\mathrm{1s}}^p K_1 = K_1 \Sigma_{\mathrm{1s}}^p$ for $p \in \mathbb{Z}\setminus\{0\}$;\\
\noindent $(b)$ $(\Sigma_{\mathrm{1s}}^p + \Sigma_{\mathrm{1s}}^q) K_1 = K_1 (\Sigma_{\mathrm{1s}}^p + \Sigma_{\mathrm{1s}}^q)$ for $p,q \in \mathbb{Z}$ with $pq \geq 0$ and $(p,q) \neq (0,0)$.
\end{corollary} 
\vspace{5pt}
For the dual diagonal matrix $\Sigma_1$, we present the following theorem, which will be frequently used throughout the proofs.
\begin{theorem}\label{QQ9}
  Let $\Sigma_1 = \text{diag}(\sigma_1, \sigma_2, \ldots, \sigma_r)$ be a diagonal matrix such that each $ \sigma_{i} $ is a positive appreciable dual number for $i = 1,\ldots, r$.
Thus, when  $n \in \mathbb{Z}\setminus\{0\}$,
  $$ \Sigma_{1}^{n}=I_r \quad \Leftrightarrow \quad \Sigma_1=I_r.$$
\end{theorem}
\begin{proof}
  Let $\sigma_i=a_i+\epsilon b_i$ with $a_i,b_i \in \mathbb{R}$ and $a_i>0$.
 Since  $\Sigma_1^n=I_r$, we obtain that
 $$ a_i^n+ \epsilon n a_i^{n-1}b_i=1 \ \ \Leftrightarrow  \ \ a_i^n=1 \ \ \text{and} \ \ a_i^{n-1}b_i=0.$$
This implies $a_i=1$, whence $b_i=0$.
It therefore follows that $\sigma_i=1$, so $\Sigma_1=I_r$.
Conversely, the converse holds.
\end{proof}

\section{Properties of dual orthogonal tripotent matrices}\label{333}

First, we consider dual orthogonal $k$-idempotent matrices (i.e., $\hat{A}^k=\hat{A}=\hat{A}^*$, where $k \in \mathbb{Z}$ and $k\geqslant 2$).
Since $\hat{A}=\hat{A}^*$, by Lemma~\ref{L01}, there exists \( \hat{U} \in \mathbb{DQ}_n^U  \) such that 
$$\hat{U}^* \hat{A} \hat{U}=\Sigma=\mathrm{diag} (\lambda_1, \cdots, \lambda_r, 0, \cdots, 0) ,$$
where $\lambda_i \in \mathbb{D}$ and $\lambda_i \neq 0$, $i=1,2,\cdots,r$.
Furthermore, from $\hat{A}^k=\hat{A}$ we obtain that $\lambda_i^k=\lambda_i$.
Let $\lambda_i=a_i+\epsilon b_i$ with $a_i,b_i \in \mathbb{R}$. This leads to 
$$\lambda_i^k = \lambda_i  \quad \Leftrightarrow  \quad   a_i^k + \epsilon k a_i^{k-1} b_i = a_i+\epsilon b_i,$$
so $a_i^k=a_i$ and $k a_i^{k-1} b_i=b_i$. 
It follows that $a_i$ can be $\pm 1$. 
Then, according to the second equation, regardless of whether $a_i$ is $1$ or $-1$ we obtain $b_i=0$.
Hence, $\lambda_i=\pm 1$.
That is,
$$\hat{A}^k=\hat{A}=\hat{A}^* \ \ \Leftrightarrow \ \ \hat{A} =\hat{U}\mathrm{diag} (1, \ldots, 1, -1, \ldots, -1,0, \ldots, 0)\hat{U}^* .$$
It follows that
\begin{itemize}
    \item if $k$ is even, then $\hat{A}$ is dual unitarily similar to $\operatorname{diag}(1,\dots,1,0,\dots,0)$;
    \item if $k$ is odd, then $\hat{A}$ is dual unitarily similar to $\operatorname{diag}(1,\dots,1,-1,\dots,-1,0,\dots,0)$.
\end{itemize}

Accordingly, the study of dual orthogonal $k$-idempotent matrices  reduces to that of dual orthogonal tripotent matrices.
It is straightforward to verify that the following theorem holds.

\begin{theorem}\label{L2}
  If $\hat{A}$ is a dual orthogonal tripotent matrix, then there exists \( \hat{U} \in \mathbb{DQ}_n^U \) such that
\begin{align}\label{E1}
\hat{A}\in\mathbb{DQ}_n^{3\text{-}OP} \ \ \Leftrightarrow \ \ \hat{A} = \hat{U} \, \mathrm{diag}(\underbrace{1, \ldots, 1}_{f}, \underbrace{-1, \ldots, -1}_{g}, \underbrace{0, \ldots, 0}_{h}) \hat{U}^*,
\end{align}
for some nonnegative integers $f,g, h$ such that $f+g+h=n$.
Conversely, the converse also holds.
\end{theorem}

\begin{corollary}\label{C1}
  If $\hat{A}$ is a dual orthogonal idempotent matrix, then there exists \( \hat{U} \in \mathbb{DQ}_n^U \) such that
\begin{align*}
\hat{A}\in\mathbb{DQ}_n^{OP} \ \ \Leftrightarrow \ \ \hat{A} = \hat{U} \, \mathrm{diag}(\underbrace{1, \ldots, 1}_{f}, \underbrace{0, \ldots, 0}_{h}) \hat{U}^*,
\end{align*}
for some nonnegative integers $f, h$ such that $f+h=n$. The converse also holds.
\end{corollary}

Using the dual Hartwig-Spindelb\"ock decomposition, we establish necessary and sufficient conditions for dual quaternion Hermitian matrices, dual orthogonal idempotent matrices, and dual orthogonal tripotent matrices.
\begin{theorem}\label{t1}
  Let $ \hat{A} \in \mathbb{DQ}^{n \times n}$ be given by \textup{(\ref{F1})}. Then \\
  \noindent  $(a)$ $\begin{aligned}[t]
    \hat{A}\in\mathbb{DQ}_n^{H} 
     \Leftrightarrow & L = 0,\ M=0,\ K^*\Sigma_1=\Sigma_1 K,\ N^*\Sigma_2=\Sigma_2 N;
    \end{aligned}$\\
  \noindent  $(b)$ $\begin{aligned}[t]
  \hat{A}\in\mathbb{DQ}_n^{OP} 
  &\Leftrightarrow \Sigma_1=I_r, \ \Sigma_2=0, \ L= 0,\ M=0,\ K=I_r;
\end{aligned}$\\
\noindent  $(c)$ $\begin{aligned}[t]
\hat{A}\in\mathbb{DQ}_n^{3\text{-}OP}   
&\Leftrightarrow \Sigma_1=I_r, \ \Sigma_2=0,\ L = 0,\  M=0,\ K^*=K. 
\end{aligned}$
\end{theorem}
\begin{proof}
   \noindent $(a):$
  Given $\hat{A}=\hat{A}^*$, (\ref{ab}) yields
\begin{align*}
  &\Sigma_{\mathrm{1s}}L_1=0, \quad  \Sigma_{\mathrm{1s}} L_{2}+\Sigma_{\mathrm{1d}} L_{1}=M_{1}^*\Sigma_{\mathrm{2d}}, \quad N_{1}^*\Sigma_{\mathrm{2d}}=\Sigma_{\mathrm{2d}}N_{1}, \notag \\
& K_1^*\Sigma_{\mathrm{1s}} = \Sigma_{\mathrm{1s}} K_1, \ \ K_1^*\Sigma_{\mathrm{1d}}+K_2^*\Sigma_{\mathrm{1s}}=\Sigma_{\mathrm{1d}}K_1+\Sigma_{\mathrm{1s}}K_2.  
\end{align*}
It follows that $L_1=0$.
Combining this with (\ref{S2}) yields  $M_1=0$, which in turn implies $L_2=0$.
It follows that $L=L_1+\epsilon L_2=0$, and by (\ref{ZXC}), $M=0$.
 In summary, we obtain that $L = 0,\ M=0,\ K^*\Sigma_1=\Sigma_1 K,\ N^*\Sigma_2=\Sigma_2 N$.

\noindent $(b):$ From  $\hat{A}=\hat{A}^*$, $(a)$ and (\ref{ab}), we have
\begin{align}\label{S3}
\hat{A}=\hat{U}\left(\begin{bmatrix}
\Sigma_{\mathrm{1s}}K_{1} & 0 \\
{0} & 0
\end{bmatrix}+\epsilon \begin{bmatrix}
\Sigma_{\mathrm{1d}}K_1+\Sigma_{\mathrm{1s}}K_2 & 0 \\
0 & \Sigma_{\mathrm{2d}}N_1
\end{bmatrix}\right)\hat{U}^*.
\end{align}
Combining this with (\ref{S1}), it follows that
\begin{align*}
\hat{A}^{2}&=\hat{U}\left(\begin{bmatrix}
{(\Sigma_{\mathrm{1s}}K_{1})^{2}} & 0 \\
{0} & {0}
\end{bmatrix}+\epsilon \begin{bmatrix}
\Sigma_{\mathrm{1s}}K_{1}(\Sigma_{\mathrm{1d}}K_1+\Sigma_{\mathrm{1s}}K_2)+(\Sigma_{\mathrm{1d}}K_1+\Sigma_{\mathrm{1s}}K_2)\Sigma_{\mathrm{1s}}K_{1} & 0 \\
0 & 0
\end{bmatrix}\right)\hat{U}^*\\
&=\hat{U}\left(\begin{bmatrix}
{(\Sigma_{\mathrm{1s}}K_{1})^{2}} & 0 \\
{0} & {0}
\end{bmatrix}+\epsilon \begin{bmatrix}
\Sigma_{\mathrm{1s}}K_{1}(K_1^*\Sigma_{\mathrm{1d}}+K_2^*\Sigma_{\mathrm{1s}})+(\Sigma_{\mathrm{1d}}K_1+\Sigma_{\mathrm{1s}}K_2)K_{1}^*\Sigma_{\mathrm{1s}} & 0 \\
0 & 0
\end{bmatrix}\right)\hat{U}^*\\
&=\hat{U}\left(\begin{bmatrix}
{(\Sigma_{\mathrm{1s}}K_{1})^{2}} & 0 \\
{0} & {0}
\end{bmatrix}+\epsilon \begin{bmatrix}
2\Sigma_{\mathrm{1s}}\Sigma_{\mathrm{1d}} & 0 \\
0 & 0
\end{bmatrix}\right)\hat{U}^*.
\end{align*}
 Under the condition $\hat{A}=\hat{A}^*$,  combined with $\hat{A}^2=\hat{A}$, we deduce
$$
\Sigma_{\mathrm{1s}}K_{1}=I_r, \quad 2\Sigma_{\mathrm{1s}}\Sigma_{\mathrm{1d}}=\Sigma_{\mathrm{1s}}K_2+\Sigma_{\mathrm{1d}}K_1 \ \ \text{and} \ \ \Sigma_{\mathrm{2d}}N_1=0.
$$
Since $\hat{A}=\hat{A}^*$  ensures $L_1=0$, it follows from (\ref{S2}) that $N_1N_1^*=I_{n-r}$. Together with $\Sigma_{\mathrm{2d}}N_1=0$, we conclude $\Sigma_{\mathrm{2d}}=0$.
Additionally, with $L_1=0$,  we find that $\Sigma_{\mathrm{1s}}K_{1} K_{1}^*\Sigma_{\mathrm{1s}}=I_r$, resulting in $\Sigma_{\mathrm{1s}}^2=I_r$,   and hence $\Sigma_{\mathrm{1s}}=I_r$. 
Consequently,  $K_1=K_1^*=I_r$, and 
$ K_2+\Sigma_{\mathrm{1d}}=2\Sigma_{\mathrm{1d}}$.
This yields $K_2=K_2^*=\Sigma_{\mathrm{1d}}$. 
From (\ref{S1}), we further obtain $2\Sigma_{\mathrm{1d}}=0$,  which leads to $\Sigma_{\mathrm{1d}}=0$ and $K_2=0$.
Therefore,
$$ \Sigma_1=\Sigma_{\mathrm{1s}}+\epsilon \Sigma_{\mathrm{1d}}=I_r, \ \ \Sigma_2=\epsilon \Sigma_{\mathrm{2d}}=0, \ \ 
L=L_1+\epsilon L_2=0, \ \ K=K_1+\epsilon K_2=I_r,$$
and by (\ref{ZXC}),  $M=0$.

\noindent $(c):$ 
From (\ref{F1}) and (\ref{S3}), it can be calculated that
$$
\hat{A}^3 =\hat{U}\left(\begin{bmatrix}
(\Sigma_{\mathrm{1s}}K_{1})^{3} & 0 \\
{0} & {0}
\end{bmatrix}+\epsilon \begin{bmatrix}
(\Sigma_{\mathrm{1s}} K_1)^2 (\Sigma_{\mathrm{1d}} K_1 + \Sigma_{\mathrm{1s}} K_2)+2\Sigma_{\mathrm{1s}}^2\Sigma_{\mathrm{1d}}K_1 & 0 \\
0 & 0
\end{bmatrix}\right)\hat{U}^*.
$$
Combining $\hat{A}=\hat{A}^*$ and $\hat{A}^3=\hat{A}$, we have 
\begin{align}\label{S4}
{(\Sigma_{\mathrm{1s}}K_{1})^{2}}=I_r, \quad 2\Sigma_{\mathrm{1s}}^2\Sigma_{\mathrm{1d}}K_1=0 \ \ \text{and} \ \ \Sigma_{\mathrm{2d}}N_{1}=0.
\end{align}
Recalling that  $(\Sigma_{\mathrm{1s}} K_1)^* = \Sigma_{\mathrm{1s}} K_1$ and (\ref{S1}),  the first equation in (\ref{S4}) yields $\Sigma_{\mathrm{1s}}^2=I_r$.
Thus, $\Sigma_{\mathrm{1s}}=I_r$.
Together with $(a)$, this yields $K_1=K_1^*$ and $K_2=K_2^*$.
 Furthermore, since $\Sigma_{\mathrm{1s}}=I_r$ and $K_1$ is invertible, the second equation in (\ref{S4}) gives $\Sigma_{\mathrm{1d}}=0$. 
 Finally, (\ref{S2}) and (\ref{S4}) lead to $\Sigma_{\mathrm{2d}}=0$.
Hence, we conclude that $\Sigma_1=\Sigma_{\mathrm{1s}}+\epsilon \Sigma_{\mathrm{1d}}=I_r, \ \ \Sigma_2=\epsilon \Sigma_{\mathrm{2d}}=0$,
$$ 
L=L_1+\epsilon L_2=0, \quad K=K_1+\epsilon K_2=K_1^*+\epsilon K_2^*=K^*,$$
and from (\ref{ZXC}) we have $M=0$.
\end{proof}

On this basis, we delve into the relationships among dual orthogonal tripotent matrices, dual quaternion matrices, and dual generalized inverse matrices.

\begin{theorem}\label{t3}
  Let $ \hat{A} \in \mathbb{DQ}^{n \times n}$. Then the following statements are equivalent: \\
\noindent $(a)$  $\hat{A}$ is a dual orthogonal tripotent matrix;\quad
\noindent $(b)$ $\hat{A} = \hat{A}^* = \hat{A}\sp{\scriptscriptstyle N}$;\\
\noindent $(c)$  $\hat{A}_e = \hat{A}^* = \hat{A}\sp{\scriptscriptstyle N}$;\quad \quad \quad \quad  \quad \quad \quad \quad \quad \quad \quad \ 
\noindent $(d)$  $\hat{A} = \hat{A}^* = \hat{A}\sp{\scriptscriptstyle N}=\hat{A}_e$.
\end{theorem}
\begin{proof}
The necessity is obvious by (\ref{E1}). It remains to prove $(c) \Rightarrow (a)$ only; the others are similar.\\
 \noindent $(c)\Rightarrow(a):$ Since $\hat{A}_e=\hat{A}^*$,  it follows from (\ref{F1}) and (\ref{Ee}) that
$$ L^*\Sigma_1=0, \ \  M^*\Sigma_2=\Sigma_1 L \ \ \text{and} \ \  N^*\Sigma_2=0.$$
This immediately yields $L=0$, whence   $M^*\Sigma_2=0$. Combined with  $N^*\Sigma_2=0$, this implies $\Sigma_2=0$ by (\ref{SS}).
Therefore, $\hat{A}=\hat{A}_e$ follows, so (\ref{FF1}) gives  $\hat{A}=\hat{A}_e=\hat{A}\hat{A}\sp{\scriptscriptstyle N}\hat{A}$.
From $\hat{A}\sp{\scriptscriptstyle N}=\hat{A}_e$, we further have $\hat{A}=\hat{A}\sp{\scriptscriptstyle N}$, 
which yields $\hat{A}\hat{A}\sp{\scriptscriptstyle N}\hat{A}=\hat{A}^3$  and consequently $\hat{A}^3=\hat{A}=\hat{A}^*$.
\end{proof}

\begin{corollary}\label{rr5}
  Let $ \hat{A} \in \mathbb{DQ}^{n \times n}$. Then, we have
  $$\hat{A}\in\mathbb{DQ}_n^{3\text{-}OP}  \quad \Leftrightarrow \quad \text{$\hat{A}^{\dagger}$ exists}, \ \hat{A} = \hat{A}^* = \hat{A}^{\dagger}.$$
\end{corollary}

\section{Characterizations of dual orthogonal tripotent matrices}\label{444}

By Theorem~\ref{t3} $(d)$, $\hat{A} \in \mathbb{DQ}_n^{3\text{-}OP}$ if and only if $\hat{A} = \hat{A}^* = \hat{A}\sp{\scriptscriptstyle N}=\hat{A}_e$.
We therefore consider the feasibility of characterizing tripotent matrices by means of averages involving $\hat{A}$,  $\hat{A}^*$, $\hat{A}\sp{\scriptscriptstyle N}$ and  $\hat{A}_e$.
\begin{theorem}\label{t5}
 Let $ \hat{A} \in \mathbb{DQ}^{n \times n}$. Then the following statements are equivalent: \\
\noindent $(a)$  $\hat{A}$ is a dual orthogonal tripotent matrix;\\
\noindent $(b)$ $\frac{1}{4}(\hat{A} + \hat{A}^*+\hat{A}\sp{\scriptscriptstyle N}+\hat{A}_e) = \hat{A}$;\\
\noindent $(c)$ $\frac{1}{4}(\hat{A} + \hat{A}^*+\hat{A}\sp{\scriptscriptstyle N}+\hat{A}_e) = \hat{A}_e$.
\end{theorem}
\begin{proof}
By Theorem \ref{t3}, we have that $(a)$ implies each of the conditions $(b)$--$(c)$. We will now prove that $(b) \Rightarrow (a)$, with the remaining cases following similarly.
  
   \noindent $(b)\Rightarrow(a):$ We have that $\hat{A}^*+\hat{A}\sp{\scriptscriptstyle N}+\hat{A}_e=3\hat{A}$. From (\ref{ab}), (\ref{abbe}) and (\ref{NN}), we obtain  
   \begin{align}
&2\Sigma_{\mathrm{1s}}L_1=0, \quad K_1^*(\Sigma_{\mathrm{1s}}+\Sigma_{\mathrm{1s}}^{-1})=2\Sigma_{\mathrm{1s}}K_1; \label{WW1}\\
&2(\Sigma_{\mathrm{1s}}L_2+\Sigma_{\mathrm{1d}}L_1)=M_1^*\Sigma_{\mathrm{2d}},  \ \ 3\Sigma_{\mathrm{2d}}N_1=N_1^*\Sigma_{\mathrm{2d}};\label{WW3} \\
&K_1^*(\Sigma_{\mathrm{1d}}-\Sigma_{\mathrm{1s}}^{-2}\Sigma_{\mathrm{1d}})+K_2^*(\Sigma_{\mathrm{1s}}+\Sigma_{\mathrm{1s}}^{-1}) =2(\Sigma_{\mathrm{1s}}K_2+\Sigma_{\mathrm{1d}}K_1). \label{WW2}
\end{align}
It follows from (\ref{WW1}) that $L_1=0$, and thus
 $2K_1^*(\Sigma_{\mathrm{1s}}+\Sigma_{\mathrm{1s}}^{-1})\Sigma_{\mathrm{1s}}K_1=2\Sigma_{\mathrm{1s}}K_1K_1^*(\Sigma_{\mathrm{1s}}+\Sigma_{\mathrm{1s}}^{-1})\Sigma_{\mathrm{1s}}K_1$, i.e., $(\Sigma_{\mathrm{1s}}^2+I_r)K_1=K_1(\Sigma_{\mathrm{1s}}^2+I_r)$.
Corollary~\ref{C2.7} yields $K_1\Sigma_{\mathrm{1s}}=\Sigma_{\mathrm{1s}} K_1$,  which together with (\ref{WW1}) gives
  $$ 2K_1^2=I_r+\Sigma_{\mathrm{1s}}^{-2} \ \ \Rightarrow \ \ 4I_r=4K_1^2(K_1^*)^2=(I_r+\Sigma_{\mathrm{1s}}^{-2})^2,$$
  implying  $(\Sigma_{\mathrm{1s}}^{-2}+3I_r)(\Sigma_{\mathrm{1s}}^{-2}-I_r)=0$. 
  As $\Sigma_{\mathrm{1s}}^{-2}+3I_r$ is invertible, $\Sigma_{\mathrm{1s}}^{-2}=I_r$ implies 
 $\Sigma_{\mathrm{1s}}=I_r$,  which combined with (\ref{WW1}) gives $K_1^*=K_1$.

 From the above, we have $L_1=0$. By (\ref{ZXC2}), it follows that $M_1=0$ and $N_1N_1^*=I_{n-r}$.  Hence, in (\ref{WW3}), it is clear that $L_2=0$. 
Moreover, $\Sigma_{\mathrm{2d}}N_1=3N_1^*\Sigma_{\mathrm{2d}}=9\Sigma_{\mathrm{2d}}N_1$, i.e., $8\Sigma_{\mathrm{2d}}N_1=0$,
which implies $\Sigma_{\mathrm{2d}}=0$.

We now consider (\ref{WW2}). Using $\Sigma_{1}=I_r$ and $K_1^*=K_1$, we have
  \begin{align}\label{WW4}
    K_2^*-K_2=\Sigma_{\mathrm{1d}}K_1.
  \end{align}
  Adding \eqref{WW4} to its conjugate transpose gives $\Sigma_{\mathrm{1d}}K_1=-K_1\Sigma_{\mathrm{1d}}$.
 Then \eqref{WW4} can be rewritten as $K_2-K_2^*=K_1\Sigma_{\mathrm{1d}}$.
By multiplying this equation on the left by $K_1$, and \eqref{WW4}  on the right by $K_1$, we obtain
  $$ K_1K_2-K_1K_2^*=\Sigma_{\mathrm{1d}} \ \ \text{and} \ \  K_2^*K_1-K_2K_1 =\Sigma_{\mathrm{1d}}.$$
From (\ref{ZXC1}), it follows that $K_2^*=-K_1K_2K_1$, which implies $K_1K_2^*=-K_2K_1$ and $K_2^*K_1=-K_1K_2$.
Then the above obtained relations lead to
$$ K_1K_2+K_1K_2=\Sigma_{\mathrm{1d}} \ \ \text{and} \ \ -K_1K_2-K_1K_2=\Sigma_{\mathrm{1d}}.$$
Adding these two equations yields  $\Sigma_{\mathrm{1d}}=0$, which together with (\ref{WW4}) implies   $K_2^*=K_2$.

In summary, we obtain $\Sigma_1=\Sigma_{\mathrm{1s}}+\epsilon \Sigma_{\mathrm{1d}}=I_r, \ \ \Sigma_2=\epsilon \Sigma_{\mathrm{2d}}=0$, 
$$ 
L=L_1+\epsilon L_2=0, \quad K=K_1+\epsilon K_2=K_1^*+\epsilon K_2^*=K^*,$$
and  $M=0$ by (\ref{ZXC}), whence by Theorem~\ref{t1} we have $\hat{A}\in\mathbb{DQ}_n^{3\text{-}OP}$.
\end{proof}

This raises the question of whether the condition $\frac{1}{4}(\hat{A} + \hat{A}^*+\hat{A}\sp{\scriptscriptstyle N}+\hat{A}_e) = \hat{A}\sp{\scriptscriptstyle N}$  is a necessary and sufficient criterion for
$\hat{A}\in\mathbb{DQ}_n^{3\text{-}OP}$. 
 However, the following example demonstrates that this condition does not hold universally.

\begin{example}
  Let $$\hat{A}=\begin{bmatrix}
1 & \epsilon\mathbf{j} & \epsilon & -\epsilon\\
-\epsilon\mathbf{j} & 1 & -\epsilon & \epsilon\\
\epsilon & -\epsilon & 3\epsilon\mathbf{k} & 0\\
-\epsilon & \epsilon & 0 & 0
\end{bmatrix}.$$
It follows from (\ref{F1}) that there exists  $\hat{U}=\begin{bmatrix}
-\mathbf{i} - \epsilon\mathbf{j} & -\epsilon\mathbf{j} & \epsilon\mathbf{k} & \epsilon \\
\epsilon\mathbf{i} & -\mathbf{j} - \epsilon\mathbf{k} & -\epsilon\mathbf{k} & -\epsilon \\
-\epsilon\mathbf{i} & \epsilon\mathbf{j} & -\mathbf{k} - \epsilon\mathbf{i} & \epsilon \\
\epsilon\mathbf{i} & -\epsilon\mathbf{j} & \epsilon\mathbf{k} & 1 - \epsilon\mathbf{i}
\end{bmatrix}$ satisfying
 $$ \hat{U}^*\hat{A}\hat{U}=\begin{bmatrix}
\Sigma_1 K & \Sigma_1 L\\
\Sigma_2 M & \Sigma_2 N
\end{bmatrix}=\begin{bmatrix}
1 & \epsilon \mathbf{i} & 0 & 0\\
-\epsilon\mathbf{i} & 1 & 0 & 0 \\
0 & 0 & 3\epsilon\mathbf{k} & 0\\
0 & 0 & 0 & 0
\end{bmatrix},$$
where $$
\Sigma_1=\begin{bmatrix}
    1 & 0\\
    0 & 1
\end{bmatrix}, \ \
\Sigma_2=\begin{bmatrix}
    3\epsilon & 0\\
    0 & 0
\end{bmatrix} \ \ \text{and} \ \ 
K=\begin{bmatrix}
1 & \epsilon\mathbf{i}\\
-\epsilon\mathbf{i} & 1
\end{bmatrix}, \ \
L=\begin{bmatrix}
    0 & 0\\
    0 & 0
\end{bmatrix},\ \ 
M=\begin{bmatrix}
    0 & 0\\
    0 & 0
\end{bmatrix}, \ \ 
  N=\begin{bmatrix}
\mathbf{k}+\epsilon\mathbf{i} & 0\\
0 & \mathbf{i}+\epsilon\mathbf{j}
\end{bmatrix}$$ satisfy (\ref{SS}).
Then, we have $\hat{U}^*\hat{A}_e \hat{U}=
\begin{bmatrix}
\Sigma_1 K & \Sigma_1 L\\
0 & 0
\end{bmatrix}=\begin{bmatrix}
 1 & \epsilon\mathbf{i} & 0 & 0\\
-\epsilon\mathbf{i} & 1 & 0 & 0 \\
0 & 0 & 0 & 0\\
0 & 0 & 0 & 0
\end{bmatrix}$ and
\begin{align*} 
  \hat{U}^*\hat{A}^* \hat{U} = \begin{bmatrix}
K^*\Sigma_1  & L^*\Sigma_1 \\
M^*\Sigma_2  & N^*\Sigma_2 
\end{bmatrix}=\begin{bmatrix}
1 & \epsilon\mathbf{i} & 0 & 0\\
-\epsilon\mathbf{i} & 1 & 0 & 0 \\
0 & 0 & -3\epsilon\mathbf{k} & 0\\
0 & 0 & 0 & 0
\end{bmatrix}, \ \
&\hat{U}^*\hat{A}\sp{\scriptscriptstyle N} \hat{U}= \begin{bmatrix}
K^*\Sigma_1^{-1}  & 0 \\
L^*\Sigma_1^{-1}  & 0\end{bmatrix}=\begin{bmatrix}
1 & \epsilon\mathbf{i} & 0 & 0\\
-\epsilon\mathbf{i} & 1 & 0 & 0 \\
0 & 0 & 0 & 0 \\
0 & 0 & 0 & 0
\end{bmatrix}.
\end{align*}
Therefore, $\hat{A} + \hat{A}^*+\hat{A}_e=3\hat{A}\sp{\scriptscriptstyle N}$, i.e., $\frac{1}{4}(\hat{A} + \hat{A}^*+\hat{A}\sp{\scriptscriptstyle N}+\hat{A}_e) = \hat{A}\sp{\scriptscriptstyle N}$.
However, this only implies $ \hat{A}_e = \hat{A}\sp{\scriptscriptstyle N}$,  while $\hat{A}\neq \hat{A}^* \neq \hat{A}_e$. 
This contradicts Theorem~\ref{t3}. Hence, $\hat{A} \notin \mathbb{DQ}^{3\text{-}OP}$.
\end{example}
In the next theorem, we discuss the relationship beetween $\frac{1}{4}(\hat{A} + \hat{A}^*+\hat{A}\sp{\scriptscriptstyle N}+\hat{A}_e) = \hat{A}\sp{\scriptscriptstyle N}$ and $\hat{A}\in\mathbb{DQ}_n^{3\text{-}OP}$.

\begin{theorem}\label{tr1}
 Let $ \hat{A} \in \mathbb{DQ}^{n \times n}$. Then the following statements are equivalent: \\
\noindent $(a)$  $\hat{A}$ is a dual orthogonal tripotent matrix;\\
\noindent $(b)$ $\frac{1}{4}(\hat{A} + \hat{A}^*+\hat{A}\sp{\scriptscriptstyle N}+\hat{A}_e) = \hat{A}\sp{\scriptscriptstyle N}$, and $\hat{A}=\hat{A}_e$ (i.e., $\hat{A}^{\dagger}$ exists).
\end{theorem}
\begin{proof}
Combining (\ref{F1}) and (\ref{Ee}) with $\hat{A}=\hat{A}_e$, we have $\Sigma_2=0$.
From (\ref{FQ}) and $\frac{1}{4}(\hat{A} + \hat{A}^*+\hat{A}\sp{\scriptscriptstyle N}+\hat{A}_e) = \hat{A}\sp{\scriptscriptstyle N}$,  we thus obtain
\begin{align*}
  L^*\Sigma_1=0 \ \ \text{and} \ \  2\Sigma_1 K=K^*(3\Sigma_1^{-1}-\Sigma_1).
\end{align*}
It is readily seen that $L=0$, and from (\ref{ZXC}) it follows that $M=0$. 
Moreover, we deduce that $2\Sigma_1 KK^*(3\Sigma_1^{-1}-\Sigma_1)=2K^*(3\Sigma_1^{-1}-\Sigma_1)\Sigma_1 K$, which in turn implies $(3I_r-\Sigma_1^2)K=K(3I_r-\Sigma_1^2)$.
By Theorem~\ref{ll}, $\Sigma_1 K=K\Sigma_1$, whence we conclude that $2K^2=3\Sigma_1^{-2}-I_r$.
Then, we have 
$$(3\Sigma_1^{-2}-I_r)^2=4K^2(K^*)^2=4I_r \ \ \Rightarrow \ \ -3\Sigma_1^{-4}(\Sigma_1-I_r)(\Sigma_1+I_r)(\Sigma_1^2+3I_r)=0.$$
By the invertibility of $\Sigma_1+I_r$ and $\Sigma_1^2+3I_r$, we deduce that $\Sigma_1=I_r$, which further implies $K=K^*$.
By Theorem~\ref {t1}, $\hat{A}\in\mathbb{DQ}_n^{3\text{-}OP}$.
\end{proof}

Similarly, we next examine whether $\frac{1}{4}(\hat{A} + \hat{A}^*+\hat{A}\sp{\scriptscriptstyle N}+\hat{A}_e) = \hat{A}^*$  is equivalent to $\hat{A}\in\mathbb{DQ}_n^{3\text{-}OP}$.
Once again, the following example shows that this equivalence is not valid in general.

\begin{example}
  Let $$\hat{A}=\begin{bmatrix}
1 & -\epsilon\left(1 + \frac{\mathbf{j}}{\sqrt{5}}\right) & \epsilon\left(1 + \frac{\mathbf{k}}{\sqrt{5}}\right) & -\epsilon \\[5pt] 
-\epsilon & -\frac{\mathbf{j}}{\sqrt{5}} + \epsilon \frac{\mathbf{k}}{\sqrt{5}} & -\epsilon \frac{\mathbf{k}}{\sqrt{5}} & 0 \\[5pt]
\epsilon & \epsilon \frac{\mathbf{j}}{\sqrt{5}} & -\frac{\mathbf{k}}{\sqrt{5}} - \epsilon \frac{\mathbf{i}}{\sqrt{5}} & 0 \\[5pt]
-\epsilon & -\epsilon \frac{\mathbf{j}}{\sqrt{5}} & \epsilon \frac{\mathbf{k}}{\sqrt{5}} & 0
\end{bmatrix}.$$
  
\noindent Then according to (\ref{F1}), there exists  $\hat{U}=\begin{bmatrix}
-\mathbf{i} - \epsilon\mathbf{j} & -\epsilon\mathbf{j} & \epsilon\mathbf{k} & \epsilon \\
\epsilon\mathbf{i} & -\mathbf{j} - \epsilon\mathbf{k} & -\epsilon\mathbf{k} & -\epsilon \\
-\epsilon\mathbf{i} & \epsilon\mathbf{j} & -\mathbf{k} - \epsilon\mathbf{i} & \epsilon \\
\epsilon\mathbf{i} & -\epsilon\mathbf{j} & \epsilon\mathbf{k} & 1 - \epsilon\mathbf{i}
\end{bmatrix}$ such that
 $$ \hat{U}^*\hat{A}\hat{U}=\begin{bmatrix}
\Sigma_1 K & \Sigma_1 L\\
\Sigma_2 M & \Sigma_2 N
\end{bmatrix}=\begin{bmatrix}
1  & 0 & 0 & 0 \\
0 & \frac{1}{\sqrt{5}}(\mathbf{j}+\epsilon\mathbf{k}) & 0 & 0 \\
0 & 0 & \frac{1}{\sqrt{5}}(\mathbf{k}+\epsilon\mathbf{i}) & 0 \\
0 & 0 & 0 & 0 
\end{bmatrix},$$
where  $\Sigma_1=\begin{bmatrix}
    1 & 0 & 0  \\
    0 & \frac{1}{\sqrt{5}} & 0 \\
    0 & 0 & \frac{1}{\sqrt{5}} \end{bmatrix}$, $\Sigma_2=0$ and $K=
\begin{bmatrix}
1 & 0 & 0 \\  
0 & \mathbf{j}+\epsilon\mathbf{k} & 0 \\
0 & 0 & \mathbf{k}+\epsilon\mathbf{i}
\end{bmatrix}$, $L=\begin{bmatrix}
  0 \\ 0 \\ 0
\end{bmatrix}$ satisfy (\ref{SS}).
Thus, since $\Sigma_2=0$, we have $\hat{U}^*\hat{A}_e \hat{U}=\hat{U}^*\hat{A} \hat{U}$ and
\begin{align*}
&\hat{U}^*\hat{A}^* \hat{U} = \begin{bmatrix}
K^*\Sigma_1  & L^*\Sigma_1 \\
M^*\Sigma_2  & N^*\Sigma_2 
\end{bmatrix}=\begin{bmatrix}
1  & 0 & 0 & 0 \\
0 & -\frac{1}{\sqrt{5}}(\mathbf{j}+\epsilon\mathbf{k}) & 0 & 0 \\
0 & 0 & -\frac{1}{\sqrt{5}}(\mathbf{k}+\epsilon\mathbf{i}) & 0 \\
0 & 0 & 0 & 0 
\end{bmatrix},\\
&\hat{U}^*\hat{A}\sp{\scriptscriptstyle N} \hat{U}= \begin{bmatrix}
K^*\Sigma_1^{-1}  & 0 \\
L^*\Sigma_1^{-1}  & 0\end{bmatrix}=\begin{bmatrix}
1  & 0 & 0 & 0 \\
0 & -\sqrt{5}(\mathbf{j}+\epsilon\mathbf{k}) & 0 & 0 \\
0 & 0 & -\sqrt{5}(\mathbf{k}+\epsilon\mathbf{i}) & 0 \\
0 & 0 & 0 & 0 
\end{bmatrix}.
\end{align*}
Thus, $\hat{A}+\hat{A}_e+\hat{A}\sp{\scriptscriptstyle N}=3\hat{A}^*$, i.e., 
$\frac{1}{4}(\hat{A} + \hat{A}^*+\hat{A}\sp{\scriptscriptstyle N}+\hat{A}_e) = \hat{A}^*$.
However, since $\hat{A}=\hat{A}_e$ but $\hat{A}\neq \hat{A}^* \neq \hat{A}\sp{\scriptscriptstyle N}$, 
Theorem~\ref{t3} implies that $\hat{A}$ is not a dual orthogonal tripotent matrix.
\end{example}

Accordingly, studying $\frac{1}{4}(\hat{A} + \hat{A}^*+\hat{A}\sp{\scriptscriptstyle N}+\hat{A}_e) = \hat{A}^*$ and 
$\hat{A}\in\mathbb{DQ}_n^{3\text{-}OP}$  yields the following result.

\begin{theorem}\label{4pj}
  Let $ \hat{A} \in \mathbb{DQ}^{n \times n}$. Then the following statements are equivalent: \\
\noindent $(a)$  $\hat{A}$ is a dual orthogonal tripotent matrix;\\
 \noindent $(b)$ $\frac{1}{4}(\hat{A} + \hat{A}^*+\hat{A}\sp{\scriptscriptstyle N}+\hat{A}_e) = \hat{A}^*$,  and $\frac{1}{5}$ is not an eigenvalue of $\hat{A}\hat{A}^*$.
\end{theorem}
\begin{proof}
   We have that $\hat{A} + \hat{A}_e + \hat{A}\sp{\scriptscriptstyle N}=3\hat{A}^*$. If we assume that $\hat{A}$ is given by $(\ref{ab})$, then by (\ref{abbe}) and $(\ref{NN})$ we have  
\begin{align}
          & 2 \Sigma_{\mathrm{1s}}L_1 = 0, \quad 2\Sigma_{\mathrm{1s}}K_1+K_1^*\Sigma_{\mathrm{1s}}^{-1}=3K_1^*\Sigma_{\mathrm{1s}}; \label{QQ1}\\
          & 3(L_1^*\Sigma_{\mathrm{1d}}+L_2^*\Sigma_{\mathrm{1s}})=\Sigma_{\mathrm{2d}}M_1+L_2^*\Sigma_{\mathrm{1s}}^{-1} - L_1^*\Sigma_{\mathrm{1s}}^{-2}\Sigma_{\mathrm{1d}}, \ \ \Sigma_{\mathrm{2d}}N_1=3N_1^*\Sigma_{\mathrm{2d}}; \label{QQ3}\\
&  2(\Sigma_{\mathrm{1s}}K_2+\Sigma_{\mathrm{1d}}K_1)+K_2^*\Sigma_{\mathrm{1s}}^{-1} - K_1^*\Sigma_{\mathrm{1s}}^{-2}\Sigma_{\mathrm{1d}} = 3(K_2^*\Sigma_{\mathrm{1s}}+K_1^*\Sigma_{\mathrm{1d}}). \label{QQ2}
        \end{align}
From (\ref{QQ1}), we obtain $L_1=0$. Hence, $2\Sigma_{\mathrm{1s}}K_1=K_1^*(3\Sigma_{\mathrm{1s}}-\Sigma_{\mathrm{1s}}^{-1})$, which implies  $ 2\Sigma_{\mathrm{1s}}K_1K_1^*(3\Sigma_{\mathrm{1s}}-\Sigma_{\mathrm{1s}}^{-1})=K_1^*(3\Sigma_{\mathrm{1s}}-\Sigma_{\mathrm{1s}}^{-1})2\Sigma_{\mathrm{1s}}K_1$, 
i.e., $K_1^*(3\Sigma_{\mathrm{1s}}^{2}-I_r)K=3\Sigma_{\mathrm{1s}}^{2}-I_r$. By Corollary~\ref{C2.7},  $K_1\Sigma_{\mathrm{1s}}=\Sigma_{\mathrm{1s}} K_1$. Combining these with (\ref{QQ1}), we conclude 
$$ 2K_1^2=3I_r-\Sigma_{\mathrm{1s}}^{-2} \quad \Rightarrow \quad 4I_r=4K_1^2(K_1^*)^2=(3I_r-\Sigma_{\mathrm{1s}}^{-2})^2,$$
whence $\Sigma_{\mathrm{1s}}^{-4}-6\Sigma_{\mathrm{1s}}^{-2}+5I_r=0$.
Let $\lambda_i=\lambda_{i,s}+\epsilon \lambda_{i,d}$, which $\lambda_{i,s},\lambda_{i,d} \in \mathbb{R}$.
Thus, we have 
$$ \lambda_i^{-4} -6\lambda_i^{-2}+5=0  \ \ \Leftrightarrow  \ \  \lambda_{i,s}^{-4}-6\lambda_{i,s}^{-2}+5=0 \ \ \text{and}  \ \ 4\lambda_{i,d}(3\lambda_{i,s}^{-3}-\lambda_{i,s}^{-5})=0.$$
It follows that $\lambda_{i,s}=1$ or $\lambda_{i,s}=\frac{1}{\sqrt{5}}$. 
Furthermore, we note that $\lambda_{i,d}=0$ regardless of whether $\lambda_{i,s}=1$ or $\lambda_{i,s}=\frac{1}{\sqrt{5}}$.
Since $\frac{1}{5}$  is not an eigenvalue of $\hat{A}\hat{A}^*$, it follows that $\lambda_{i,s}\neq \frac{1}{\sqrt{5}}$.
Thus, we must have $\lambda_{i,s}=1$, which further implies $\Sigma_{\mathrm{1s}}=I_r$.
By (\ref{QQ1}), we therefore conclude that $K_1=K_1^*$.

Additionally, $L_1=0$ implies $M_1=0$.
Then, combined with  $\Sigma_{\mathrm{1s}}=I_r$, it follows from (\ref{QQ3}) that  $L_2=0$. 
Since $\Sigma_{\mathrm{2d}}N_1=3N_1^*\Sigma_{\mathrm{2d}}$, we deduce $N_1^*\Sigma_{\mathrm{2d}}=3\Sigma_{\mathrm{2d}}N_1$,
and thus $\Sigma_{\mathrm{2d}}N_1=3N_1^*\Sigma_{\mathrm{2d}}=9\Sigma_{\mathrm{2d}}N_1$, i.e., $8\Sigma_{\mathrm{2d}}N_1=0$.
 Together with $L_1=0$, this yields $N_1N_1^*=I_{n-r}$,  whence $\Sigma_{\mathrm{2d}}=0$.

Considering (\ref{QQ2}) together with $\Sigma_{\mathrm{1s}}=I_r$, we obtain
$$ 2(K_2+\Sigma_{\mathrm{1d}}K_1)+K_2^* - K_1^*\Sigma_{\mathrm{1d}} = 3(K_2^*+K_1^*\Sigma_{\mathrm{1d}}) \ \ \Rightarrow \ \ K_2-K_2^*=2K_1^*\Sigma_{\mathrm{1d}}-\Sigma_{\mathrm{1d}}K_1.$$
Adding this to its conjugate transpose gives $K_1^*\Sigma_{\mathrm{1d}}=-\Sigma_{\mathrm{1d}}K_1$. 
Thus, similar to Theorem~\ref{t5} (\ref{WW4}), we obtain $\Sigma_{\mathrm{1d}}=0$ and $K_2=K_2^*$.

From the above, $\Sigma_1=\Sigma_{\mathrm{1s}}+\epsilon \Sigma_{\mathrm{1d}}=I_r, \ \ \Sigma_2=\epsilon \Sigma_{\mathrm{2d}}=0$,
$$ L=L_1+\epsilon L_2=0, \quad  K=K_1+\epsilon K_2=K_1^*+\epsilon K_2^*=K^*,$$
and from (\ref{ZXC}), $M=0$.
Hence, by Theorem~\ref{t1}, $\hat{A}\in\mathbb{DQ}_n^{3\text{-}OP}$.
\end{proof} 

Motivated by Theorem~\ref{t5} $(b)$, it is natural to investigate whether the mean equality for three such elements provides an effective characterization of dual orthogonal tripotent matrices.
\begin{theorem}\label{t4}  
  Let $ \hat{A} \in \mathbb{DQ}^{n \times n}$. Then the following statements are equivalent: \\
\noindent $(a)$  $\hat{A}$ is a dual orthogonal tripotent matrix;\\
\noindent $(b)$ $\frac{1}{3}(\hat{A} + \hat{A}^*+\hat{A}\sp{\scriptscriptstyle N}) = \hat{A}$;\\
\noindent $(c)$ $\frac{1}{3}(\hat{A}_e + \hat{A}^*+\hat{A}\sp{\scriptscriptstyle N}) = \hat{A}$;\\
\noindent $(d)$  $\frac{1}{3}(\hat{A}_e + \hat{A}^*+\hat{A}\sp{\scriptscriptstyle N}) = A_e$;\\
\noindent $(e)$ $\frac{1}{3}(\hat{A} + \hat{A}^*+\hat{A}\sp{\scriptscriptstyle N}) = \hat{A}\sp{\scriptscriptstyle N}$, and $\hat{A}=\hat{A}_e$;\\
\noindent $(f)$ $\frac{1}{3}(\hat{A}_e + \hat{A}^*+\hat{A}\sp{\scriptscriptstyle N}) = \hat{A}\sp{\scriptscriptstyle N}$, and $\hat{A}=\hat{A}_e$;\\
\noindent $(g)$ $\frac{1}{3}(\hat{A} + \hat{A}^*+\hat{A}\sp{\scriptscriptstyle N}) = \hat{A}^*$,  and $\frac{1}{3}$ is not an eigenvalue of $\hat{A}\hat{A}^*$;\\
\noindent $(h)$ $\frac{1}{3}(\hat{A}_e + \hat{A}^*+\hat{A}\sp{\scriptscriptstyle N}) = \hat{A}^*$,  and  $\frac{1}{3}$ is not an eigenvalue of $\hat{A}\hat{A}^*$.
\end{theorem}

\begin{proof}
Similarly to Theorem~\ref{t5} and Theorem~\ref{4pj}, the proof is omitted.
\end{proof}

Likewise, we may also characterize dual orthogonal tripotent matrices by constructing mean equations formulated in terms of  $\hat{A}$, $\hat{A}_e$, $\hat{A}^*$, $\hat{A}\sp{\scriptscriptstyle N}$ as well as their products. 

\begin{theorem}\label{2pj}
  Let $ \hat{A} \in \mathbb{DQ}^{n \times n}$. Then the following statements are equivalent: \\
\noindent $(a)$  $\hat{A}$ is a dual orthogonal tripotent matrix;\quad \ \ \
\noindent $(b)$  $\frac{1}{2}(\hat{A}+\hat{A}\hat{A}^*\hat{A})=\hat{A}\sp{\scriptscriptstyle N}$;\\
\noindent $(c)$  $\frac{1}{2}(\hat{A}+\hat{A}\sp{\scriptscriptstyle N}\hat{A}\hat{A}^*)=\hat{A}\sp{\scriptscriptstyle N}$;\   \ \quad \quad \quad \quad \quad \quad \quad \quad \quad
\noindent $(d)$  $\frac{1}{2}(\hat{A}\sp{\scriptscriptstyle N}+\hat{A}\sp{\scriptscriptstyle N}\hat{A}\hat{A}^*)=\hat{A}$;\\
\noindent $(e)$  $\frac{1}{2}(\hat{A}+\hat{A}^2\hat{A}^*)=\hat{A}\sp{\scriptscriptstyle N}$; \    \quad \quad \quad \quad \quad \quad \quad \quad \quad \ \
\noindent $(f)$  $\frac{1}{2}(\hat{A}+\hat{A}\hat{A}^*\hat{A}\sp{\scriptscriptstyle N})=\hat{A}\sp{\scriptscriptstyle N}$;\\
\noindent $(g)$  $\frac{1}{2}(\hat{A}\sp{\scriptscriptstyle N}+\hat{A}\hat{A}^*\hat{A}\sp{\scriptscriptstyle N})=\hat{A}$.
\end{theorem}
\begin{proof}
  By Theorem \ref{t3} together with (\ref{F1}) and (\ref{FQ}),  condition $(a)$ implies all of $(b)$--$(g)$. We prove only the implication $(b) \Rightarrow (a)$; the remaining cases follow similarly. \\
 \noindent $(b)\Rightarrow(a):$  From (\ref{F1}) and (\ref{FQ}), together with $\hat{A}+\hat{A}\hat{A}^*\hat{A}=2\hat{A}\sp{\scriptscriptstyle N}$, we obtain
 \begin{align}
  &(\Sigma_1+\Sigma_1^3)K=2K^*\Sigma_1^{-1}, \quad  (\Sigma_1+\Sigma_1^3)L=0, \label{FF11} \\
  &(\Sigma_2+\Sigma_2^3) M=2L^*\Sigma_1^{-1}, \quad  (\Sigma_2+\Sigma_2^3) N=0. \notag
 \end{align}
 Note that $\Sigma_2$ is an infinitesimal dual diagonal matrix with all diagonal entries positive, which implies $\Sigma_2^3=0$.
 Then the second identity in (\ref{FF11}) yields  $L=0$, which further implies $\Sigma_2 M=0$.  
 Consequently, it follows readily from (\ref{SS}) and $\Sigma_2 N=0$ that  $\Sigma_2=0$.
 We next turn to the first identity in (\ref{FF11}), which gives $2(\Sigma_1+\Sigma_1^3)KK^*\Sigma_1^{-1}=2K^*\Sigma_1^{-1}(\Sigma_1+\Sigma_1^3)K$, i.e., $(I_r+\Sigma_1^2)K=K(I_r+\Sigma_1^2)$.
 Then Theorem~\ref{ll} implies $\Sigma_1 K=K\Sigma_1$. 
Combining this with the first identity in (\ref{FF11}), we obtain $ \Sigma_1^2+\Sigma_1^4=2K^2$.
Accordingly, $(\Sigma_1^2+\Sigma_1^4)^2=4K^2(K^*)^2=4I_r$, so
  $$  (\Sigma_1-I_r)(\Sigma_1+I_r)(\Sigma_1^2+2I_r)(\Sigma_1^4+\Sigma_1^2+2I_r)=0.$$
Since $\Sigma_1+I_r$, $\Sigma_1^2+2I_r$ and $\Sigma_1^4+\Sigma_1^2+2I_r$ are invertible, we have  $\Sigma_1=I_r$. 
Together with (\ref{FF11}), this further yields $K=K^*$. 
By (\ref{ZXC}), $L=0$ implies $M=0$, and thus by Theorem~\ref{t1}, we have $\hat{A}\in\mathbb{DQ}_n^{3\text{-}OP}$. 
\end{proof}

\begin{theorem}
  Let $ \hat{A} \in \mathbb{DQ}^{n \times n}$. Then the following statements are equivalent: \\
\noindent $(a)$  $\hat{A}$ is a dual orthogonal tripotent matrix;\ \ \
\noindent $(b)$ $\frac{1}{3}(\hat{A}+\hat{A}\sp{\scriptscriptstyle N}+\hat{A}\hat{A}\sp{\scriptscriptstyle N}\hat{A}^*)= \hat{A}_e$;\\
\noindent $(c)$ $\frac{1}{3}(\hat{A}+\hat{A}\sp{\scriptscriptstyle N}+\hat{A}\hat{A}^*\hat{A}\sp{\scriptscriptstyle N})= \hat{A}_e$;\quad \quad \quad \quad \quad \quad  \ 
\noindent $(d)$ $\frac{1}{3}(\hat{A}+\hat{A}\sp{\scriptscriptstyle N}+\hat{A}\sp{\scriptscriptstyle N}\hat{A}\hat{A}^*)= \hat{A}_e$;\\
\noindent $(e)$ $\frac{1}{3}(\hat{A}_e+\hat{A}\sp{\scriptscriptstyle N}+\hat{A}\hat{A}^*\hat{A}\sp{\scriptscriptstyle N})= \hat{A}$;\quad \quad \quad \quad \quad \quad  \ 
\noindent $(f)$ $\frac{1}{3}(\hat{A}_e+\hat{A}\sp{\scriptscriptstyle N}+\hat{A}\sp{\scriptscriptstyle N}\hat{A}\hat{A}^*)= \hat{A}$;\\
\noindent $(g)$ $\frac{1}{3}(\hat{A}+\hat{A}_e+\hat{A}\hat{A}^*\hat{A}\sp{\scriptscriptstyle N})= \hat{A}\sp{\scriptscriptstyle N}$;\quad \quad \quad \quad \quad \ \ \ \ 
\noindent $(h)$ $\frac{1}{3}(\hat{A}+\hat{A}_e+\hat{A}\hat{A}^*\hat{A})= \hat{A}\sp{\scriptscriptstyle N}$;\\
\noindent $(i)$ $\frac{1}{3}(\hat{A}+\hat{A}_e+\hat{A}\sp{\scriptscriptstyle N}\hat{A}\hat{A}^*)= \hat{A}\sp{\scriptscriptstyle N}$;\quad \quad \quad \quad \quad \quad  \ \
\noindent $(j)$ $\frac{1}{3}(\hat{A}+\hat{A}_e+\hat{A}^2\hat{A}^*)= \hat{A}\sp{\scriptscriptstyle N}$;\\
\noindent $(k)$ $\frac{1}{3}(\hat{A}^*+\hat{A}\sp{\scriptscriptstyle N}+\hat{A}\hat{A}\sp{\scriptscriptstyle N}\hat{A}^*)= \hat{A}_e$;\quad \quad \quad \quad \quad \ \ \
\noindent $(l)$ $\frac{1}{3}(\hat{A}^*+\hat{A}\sp{\scriptscriptstyle N}+\hat{A}\sp{\scriptscriptstyle N}\hat{A}\hat{A}^*)= \hat{A}_e$.
\end{theorem}
\begin{proof}
  Theorem \ref{t3} together with (\ref{F1}) and (\ref{FQ}) implies that condition $(a)$ yields each of $(b)$--$(l)$. We only prove $(b) \Rightarrow (a)$ here,  as the remaining cases are analogous. \\
  \noindent $(b)\Rightarrow(a):$ 
  Using (\ref{F1}), (\ref{Ee}), (\ref{FQ}) and  $\frac{1}{3}(\hat{A}+\hat{A}\sp{\scriptscriptstyle N}+\hat{A}\hat{A}\sp{\scriptscriptstyle N}\hat{A}^*)= \hat{A}_e$, we obtain
  $$ K^*(\Sigma_1+\Sigma_1^{-1})=2\Sigma_1 K, \ \ M^*\Sigma_2=2\Sigma_1 L, \ \ \Sigma_2 M+L^*\Sigma_1^{-1}=0 \ \ \text{and}\ \  \Sigma_2 N=0.$$
This gives $\Sigma_2 M+L^*\Sigma_1^{-1}=L^*(2\Sigma_1+\Sigma_1^{-1})=0$, so $L=0$.
It follows that $\Sigma_2 M=0$. Together with $\Sigma_2 N=0$, we readily obtain  $\Sigma_2=0$  from (\ref{SS}).
Furthermore, since $L=0$, we conclude that $2K^*(\Sigma_1+\Sigma_1^{-1})\Sigma_1 K=2\Sigma_1 KK^*(\Sigma_1+\Sigma_1^{-1})$, i.e.,
$(\Sigma_1^2+I_r)K=K(\Sigma_1^2+I_r)$, which by Theorem~\ref{ll} yields   $\Sigma_1 K=K\Sigma_1$.
Thus, from the first identity,
$$ 2K^2=I_r+\Sigma_1^{-2} \ \ \Rightarrow \ \ 4I_r=4K^2(K^*)^2= (I_r+\Sigma_1^{-2})^2, $$
which implies $(\Sigma_1^{-2}+3I_r)(\Sigma_1^{-2}-I_r)=0$.
Since $(\Sigma_1^{-2}+3I_r)$ is invertible, it follows that $\Sigma_1^{-2}=I_r$,
From Theorem~\ref{QQ9}, we obtain  $\Sigma_1=I_r$, which implies  $K=K^*$.
By (\ref{ZXC}), $L=0$ gives $M=0$.
Thus, Theorem~\ref{t1} yields $\hat{A}\in\mathbb{DQ}_n^{3\text{-}OP}$.
\end{proof}

Furthermore, we are also interested in characterizing dual orthogonal tripotent matrices by linear equations in $\hat{A}$, $\hat{A}_e$, $\hat{A}^*$, $\hat{A}\sp{\scriptscriptstyle N}$ and their products.

\begin{theorem}
  Let $ \hat{A} \in \mathbb{DQ}^{n \times n}$. Then the following statements are equivalent: \\
\noindent $(a)$  $\hat{A}$ is a dual orthogonal tripotent matrix;\ \
\noindent $(b)$  $\hat{A}+\hat{A}^*\hat{A}\hat{A}\sp{\scriptscriptstyle N}=\hat{A}\sp{\scriptscriptstyle N}+\hat{A}\hat{A}^*\hat{A}\sp{\scriptscriptstyle N}$;\\
\noindent $(c)$  $\hat{A}+\hat{A}^2\hat{A}^*=\hat{A}\sp{\scriptscriptstyle N}+\hat{A}^2\hat{A}\sp{\scriptscriptstyle N}$;\ \ \ \quad \quad \quad \quad \quad \quad 
\noindent $(d)$  $\hat{A}^*+\hat{A}\hat{A}^*\hat{A}=\hat{A}\sp{\scriptscriptstyle N}+\hat{A}\sp{\scriptscriptstyle N}\hat{A}\hat{A}^*$;\\
\noindent $(e)$  $\hat{A}^*+\hat{A}^*\hat{A}\hat{A}^*=\hat{A}\sp{\scriptscriptstyle N}+\hat{A}\sp{\scriptscriptstyle N}\hat{A}\hat{A}^*$;\ \  \ \quad \quad \quad \quad
\noindent $(f)$  $\hat{A}^*+\hat{A}^2\hat{A}^*=\hat{A}\sp{\scriptscriptstyle N}+\hat{A}\sp{\scriptscriptstyle N}\hat{A}\hat{A}^*$;\\
\noindent $(g)$  $\hat{A}^*+\hat{A}^2\hat{A}^*=\hat{A}\sp{\scriptscriptstyle N}+\hat{A}^*\hat{A}\hat{A}\sp{\scriptscriptstyle N}$.
\end{theorem}
\begin{proof}
 From Theorem \ref{t3}, $(a)$ implies all conditions $(b)$-$(g)$. We only establish $(b) \Rightarrow (a)$, the other cases are analogous.\\
 \noindent $(b)\Rightarrow(a):$  Combining (\ref{F1}) and (\ref{FQ}) with $\hat{A}+\hat{A}^*\hat{A}\hat{A}\sp{\scriptscriptstyle N}=\hat{A}\sp{\scriptscriptstyle N}+\hat{A}\hat{A}^*\hat{A}\sp{\scriptscriptstyle N}$ yields 
 \begin{align}\label{FF12} 
  &\Sigma_1 K +K^*\Sigma_1=K^*\Sigma_1^{-1}+\Sigma_1^2 K^*\Sigma_1^{-1}, \quad  \Sigma_1 L=0, \quad \Sigma_2 N=0. 
 \end{align}
By (\ref{FF12}), we have $L=0$. Then, from (\ref{ZXC}), it follows that $M=0$ and $NN^*=I_{n-r}$. Furthermore, with $\Sigma_2 N=0$, it is straightforward to see that $\Sigma_2=0$.
Moreover, subtracting the first identity in (\ref{FF12}) from its conjugate transpose gives
$$ (I_r+\Sigma_1^2)K^*\Sigma_1^{-1}=\Sigma_1^{-1}K(I_r+\Sigma_1^2) \ \ \Leftrightarrow  \ \ (\Sigma_1+\Sigma_1^3)K^*=K(\Sigma_1+\Sigma_1^3).$$
Then  $ (\Sigma_1+\Sigma_1^3)K^*K(\Sigma_1+\Sigma_1^3)=K(\Sigma_1+\Sigma_1^3) (\Sigma_1+\Sigma_1^3)K^*$, i.e., $(\Sigma_1+\Sigma_1^3)^2K=K(\Sigma_1+\Sigma_1^3)^2$. 
By Theorem~\ref{ll}, we have $\Sigma_1 K=K\Sigma_1$.
Together with (\ref{FF12}), this leads to $\Sigma_1 (K+K^*)=(\Sigma_1^{-1}+\Sigma_1)K^*$. 
Multiplying this identity by $K$ on the right and by $\Sigma_1^{-1}$ on the left yields $K^2=\Sigma_1^{-2}$,
which implies $\Sigma_1^{-4}=K^2(K^*)^2=I_r$. 
By Theorem~\ref{QQ9}, we obtain $\Sigma_1=I_r$, whence $K=K^*$.
Hence, $\hat{A}\in\mathbb{DQ}_n^{3\text{-}OP}$ holds.
\end{proof}

As \(\hat{A}^*\hat{A}\) and \(\hat{A}\hat{A}^*\)  satisfy 
$(\hat{A}^*\hat{A})^*=(\hat{A})^*(\hat{A}^*)^*=\hat{A}^*\hat{A} \ \ \text{and}\ \ (\hat{A}\hat{A}^*)^*=(\hat{A}^*)^*(\hat{A})^*=\hat{A}\hat{A}^*$,
 they are dual quaternion Hermitian matrices. 
Similary to the case of negative integer powers for complex matrices (Khatri \citeyear{Khat}),  for any $\alpha \in \mathbb{Z}$ we have  
$(\hat{A}\sp{\scriptscriptstyle N})^\alpha = (\hat{A}^\alpha)^{\text{\tiny N}}$,
which allows us to define $$\hat{A}^{-\alpha} = (\hat{A}\sp{\scriptscriptstyle N})^\alpha = (\hat{A}^\alpha)^{\text{\tiny N}}.$$
The above definition also applies to DMPGI $\hat{A}^{\dagger}$.
In what follows, we characterize dual quaternion matrices for integers \(p, q\) using integer powers of \(\hat{A}^*\hat{A}\) and \(\hat{A}\hat{A}^*\).

\begin{theorem}\label{tt5}
    Let $\hat{A}\in\mathbb{DQ}^{n\times n}$ and $p,q \in \mathbb{Z}$. If $\hat{A}^{\dagger}$ exists, the following conditions are equivalent: \\
  \noindent $(a)$  $\hat{A}$ is a dual orthogonal tripotent matrix;\\
  \noindent$(b)$ $\hat{A}(\hat{A}\hat{A}^*)^{p}=\hat{A}^{\dagger}(\hat{A}\hat{A}^*)^{q}$ with $p\neq q$ and $p - q + 1\neq0$;\\
   \noindent $(c)$ $\hat{A}(\hat{A}^*\hat{A})^{p}=\hat{A}^{\dagger}(\hat{A}^*\hat{A})^{q}$ with $p\neq q$ and $p - q + 1\neq0$;\\
  \noindent    $(d)$  $\hat{A}^*(\hat{A}\hat{A}^*)^{p}=\hat{A}(\hat{A}^*\hat{A})^{q}$ with $p\neq q$ and $p + q + 1\neq0$;\\
     \noindent  $(e)$ $\hat{A}(\hat{A}^*\hat{A})^{p}=\hat{A}^*(\hat{A}^*\hat{A})^{q}$ with $p\neq q$ and $p - q + 1\neq0$;\\
     \noindent $(f)$ $\hat{A}(\hat{A}\hat{A}^*)^{p}=\hat{A}^*(\hat{A}\hat{A}^*)^{q}$ with $p\neq q$ and $p - q - 1\neq0$;\\
    \noindent  $(g)$ $\hat{A}(\hat{A}^*\hat{A})^{p}=\hat{A}^{\dagger}(\hat{A}\hat{A}^*)^{q}$ with $p\neq -q$ and $p - q + 1\neq0$. 
 \end{theorem}
\begin{proof}
By Theorem~\ref{t3}, condition $(a)$ implies each of the conditions $(b)$--$(g)$. We will now prove that $(b)$ implies $(a)$, noting that a similar argument shows that any of $(c)$--$(g)$ also implies $(a)$.

\noindent  $(b)\Rightarrow(a):$   
By $(\ref{F1})$ and Lemma~\ref{LN}, when  $\hat{A}^{\dagger}$ exists, it follows from  $\hat{A}(\hat{A}\hat{A}^*)^{p}=\hat{A}^{\dagger}(\hat{A}\hat{A}^*)^{q}$ that
\begin{align}\label{4.15}
\Sigma_2=0, \quad L^*\Sigma_1^{2p-1} = 0 \quad \text{and} \quad \Sigma_1 K = K^{*}\Sigma_1^{2(q - p)-1}.
\end{align}
Thus, $L=0$ readily follows, implying 
$\Sigma_1 K K^{*}\Sigma_1^{2(q - p)-1}=K^{*}\Sigma_1^{2(q - p)-1} \Sigma_1 K$,
 and hence $ K\Sigma_1^{2(q - p)}=\Sigma_1^{2(q - p)}K.$ 
For $p \neq q$, by Theorem~\ref{ll} we get  $K\Sigma_1 = \Sigma_1 K$, which combined with (\ref{4.15}) gets $K\Sigma_1 = K^{*}\Sigma_1^{2(p - q)-1}$. 
        Hence, $\Sigma_1^{2(p - q + 1)} = (K^*)^2 = K^2$, so $\Sigma_1^{4(p - q + 1)} = K^2(K^*)^2 = I_r$.
By Theorem~\ref{QQ9}, if $p - q + 1 \neq 0$, then $\Sigma_1 = I_{r}$.
This yields $K^2 = I_{r}$, whence  $K = K^*$.
Then it follows that $\hat{A}\in\mathbb{DQ}_n^{3\text{-}OP}$.
      \end{proof}

It is natural to ask whether Theorem~\ref{tt5} remains valid without the DMPGI, and we find that the additional condition $\Sigma_2 = 0$ is necessary for the theorem to hold.
In what follows, we focus on the case of the NDMPI. For $p,q \in \mathbb{Z}$,
the condition $\Sigma_2 = 0$ naturally arises in calculations involving integer powers of  $\hat{A}\hat{A}^*$ or $\hat{A}^*\hat{A}$, as well as their products with  $\hat{A}$, $\hat{A}^*$ or $\hat{A}\sp{\scriptscriptstyle N}$.
Consequently, within the framework of Theorem~\ref{tt5}, setting \(p = 0\) or \(q = 0\) leads to the following results:

\begin{theorem}\label{t6}
    Let $\hat{A}\in\mathbb{DQ}^{n\times n}$ and $p \in \mathbb{Z}\setminus\{0\}$. Then the following statements are equivalent: \\
\noindent $(a)$  $\hat{A}$ is a dual orthogonal tripotent matrix;\ \
\noindent $(b)$  $\hat{A}^*=\hat{A}(\hat{A}\hat{A}^*)^p$ with  $p \neq 1$;\\
\noindent $(c)$  $\hat{A}=\hat{A}^*(\hat{A}\hat{A}^*)^p$ with $p \neq -1$;\ \ \quad \quad \quad \quad \quad
\noindent $(d)$  $\hat{A}=\hat{A}\sp{\scriptscriptstyle N}(\hat{A}\hat{A}^*)^p$ with  $p \neq 1$;\\
\noindent $(e)$  $\hat{A}^*=\hat{A}(\hat{A}^*\hat{A})^p$ with  $p \neq -1$;\ \ \quad \quad \quad \quad \quad
\noindent $(f)$  $\hat{A}=\hat{A}^*(\hat{A}^*\hat{A})^p$ with  $p \neq 1$;\\
\noindent $(g)$  $\hat{A}=\hat{A}\sp{\scriptscriptstyle N}(\hat{A}^*\hat{A})^p$ with  $p \neq 1$.
\end{theorem}
\begin{proof}
  It is clear that  $(a)$ implies $(b)$-$(g)$. We only prove $(b) \Rightarrow (a)$ and $(f)\Rightarrow (a)$, as  $(c)-(e) \Rightarrow (a)$ and $(g)\Rightarrow (a)$ 
 follow similarly, respectively.

\noindent  $(b)\Rightarrow(a):$  By $(\ref{F1})$,   $\hat{A}^*=\hat{A}(\hat{A}\hat{A}^*)^p$ is equivalent to 
   \begin{align*}
     L^*\Sigma_1=0, \quad M^*\Sigma_2=0, \quad N^*\Sigma_2=0 \quad \text{and} \quad K^*\Sigma_1=\Sigma_1K\Sigma_1^{2p}.
   \end{align*}
   Clearly, $L=0$. By (\ref{SS}), the above equation implies $\Sigma_2=0$ and 
   \begin{align}\label{W1}
   \Sigma_1^{-1}K^*=K\Sigma_1^{2p-1}.
   \end{align}
   Then $\Sigma_1^{-1}K^*K\Sigma_1^{2p-1}=K\Sigma_1^{2p-1}\Sigma_1^{-1}K^*$, i.e., $\Sigma_1^{2(p-1)}K=K\Sigma_1^{2(p-1)}$. By Theorem~\ref{ll}, $\Sigma_1K= K\Sigma_1$ for $p \neq 1$.
   Consequently, from (\ref{W1}), we get $K^2=\Sigma_1^{2p}$, whence $\Sigma_1^{4p}=K^2(K^*)^2=I_r$. 
Theorem~\ref{QQ9} yields $\Sigma_1=I_r$, which together with (\ref{W1}) implies $K=K^*$.
Further, (\ref{ZXC}) implies  $M=0$  if $L=0$.
It follows from Theorem~\ref{t1} that $\hat{A}\in\mathbb{DQ}_n^{3\text{-}OP}$.

\noindent  $(f)\Rightarrow(a):$  By $(\ref{F1})$, it follows from $\hat{A}=\hat{A}^*(\hat{A}^*\hat{A})^p$ that
   \begin{align}
    (K^*\Sigma_1K^* + M^*\Sigma_2L^*)\Sigma_1^{2p}K=\Sigma_1K, \quad (K^*\Sigma_1K^* + M^*\Sigma_2L^*)\Sigma_1^{2p}L=\Sigma_1L, \label{W2} \\ 
    (L^*\Sigma_1K^* + N^*\Sigma_2L^*)\Sigma_1^{2p}K=\Sigma_2M, \quad (L^*\Sigma_1K^* + N^*\Sigma_2L^*)\Sigma_1^{2p}L=\Sigma_2N. \label{W3}
   \end{align}
Multiplying (\ref{W3}) on the right by $M^*$ and $N^*$ respectively, we obtain $\Sigma_2=0$ from (\ref{SS}). 
Thus, multiplying (\ref{W3}) on the right by $K^*$ and $L^*$  respectively, and the same operation on (\ref{W2}) yields
\begin{align}\label{W4}
 L^*\Sigma_1K^*\Sigma_1^{2p}=0 \quad \text{and} \quad K^*\Sigma_1K^*\Sigma_1^{2p}=\Sigma_1.
  \end{align}
Repeating this operation for (\ref{W4}) gives $\Sigma_1K^*\Sigma_1^{2p}=K\Sigma_1$. 
Substituting this result into the second equation of (\ref{W4}) leads to $K^*K\Sigma_1=\Sigma_1$, which implies $K^*K=I_r$.
From the first equation in (\ref{W4}), we deduce $L^*\Sigma_1K^*=0$, which implies  $L^*\Sigma_1=0$, and hence $L=0$.
Consequently, we have
 $K^*\Sigma_1=K\Sigma_1^{1-2p}$, and an argument similar to that for $(b)\Rightarrow(a)$ shows that $\Sigma_1=I_r$ and $K=K^*$ for $p \neq 1$.
Moreover, from (\ref{ZXC}), we deduce $M=0$ whenever $L=0$.
Finally, by Theorem~\ref{t1}, $\hat{A}\in\mathbb{DQ}_n^{3\text{-}OP}$.
\end{proof}

 Theorem~\ref{t1} implies that $\hat{A}\hat{A}^* \in \mathbb{DQ}_{n}^{OP}$ when $\hat{A}\in\mathbb{DQ}_n^{3\text{-}OP}$. 
Hence, we see that this property characterizes dual orthogonal tripotent matrices. 

\begin{theorem}\label{range} Let $\hat{A} \in \mathbb{DQ}^{n \times n}$ and $p \in \mathbb{Z}\setminus\{0\}$. 
  Then the following statements are equivalent: 

  \noindent $(a)$  $\hat{A}$ is a dual orthogonal tripotent matrix;\ \ \ 
  $(b)$ $\hat{A}\hat{A}^* \in \mathbb{DQ}_{n}^{OP}$ and $\hat{A} = \hat{A}\sp{\scriptscriptstyle N}(\hat{A}\hat{A}^*)^p$;\\ 
 $(c)$ $\hat{A}\hat{A}^* \in \mathbb{DQ}_{n}^{OP}$ and  $\hat{A}^* = \hat{A}(\hat{A}^*\hat{A})^p$;\quad \quad \quad  \
  \noindent $(d)$ $\hat{A}\hat{A}^* \in \mathbb{DQ}_{n}^{OP}$ and  $\hat{A} = \hat{A}^*(\hat{A}\hat{A}^*)^p$;\\ 
  \noindent $(e)$ $\hat{A}\hat{A}^* \in \mathbb{DQ}_{n}^{OP}$ and  $\hat{A}^* = \hat{A}(\hat{A}\hat{A}^*)^p$; \quad \quad \quad  
  \noindent $(f)$  $\hat{A}\hat{A}^* \in \mathbb{DQ}_{n}^{OP}$ and $\hat{A}=\hat{A}\sp{\scriptscriptstyle N}(\hat{A}^*\hat{A})^p$;\\
  \noindent $(g)$ $\hat{A}\hat{A}^* \in \mathbb{DQ}_{n}^{OP}$ and  $\hat{A}= \hat{A}^*(\hat{A}^*\hat{A})^p$.
\end{theorem}
\begin{proof} 
 The implication from $(a)$ to $(b)$-$(g)$ is evident. The proof will be given only for the implication  $(b) \Rightarrow (a)$, as the remaining implications follow analogously.\\
\noindent  $(b)\Rightarrow(a):$ From (\ref{F1}),  $\hat{A}\hat{A}^* \in \mathbb{DQ}_{n}^{OP}$  is equivalent to $\Sigma_1^2= \Sigma_1$, and hence implies $\Sigma_1=I_r$. 
Similarly to Theorem~\ref{t6}, $\hat{A} = \hat{A}\sp{\scriptscriptstyle N}(\hat{A}\hat{A}^*)^p$ yields
$ \Sigma_2=0$, $L=0$ and $\Sigma_1 K=K^*\Sigma_1^{2p-1}$.
Combining this with  $\Sigma_1=I_r$, we obtain $K=K^*$.
Therefore,  $\hat{A}\in\mathbb{DQ}_n^{3\text{-}OP}$.
\end{proof}

\begin{remark}
In this section, all results except Theorem~\ref{tt5} are developed for the NDMPI.
It is natural to investigate whether the preceding conclusions hold for DMPGI.
However, the existence of the DMPGI imposes the additional condition $\Sigma_2=0$, which by  (\ref{F1}) and (\ref{Ee}) is equivalent to $\hat{A}=\hat{A}_e$.
Accordingly, replacing $\hat{A}\sp{\scriptscriptstyle N}$ with $\hat{A}^{\dagger}$  in the above results requires only the additional condition $\hat{A}=\hat{A}_e$, and further discussion is omitted.
\end{remark}

\section{Conclusions}\label{555}
This paper has introduced the concept of dual orthogonal tripotent matrices and derives their fundamental properties. 
The Hermiticity of such matrices inherently guarantees dual unitary diagonalizability, which underpins analyses of their structural features and links with dual generalized inverse matrices.
We derive characterizations of dual orthogonal tripotent matrices via mean-value equalities incorporating dual quaternion matrices, their dual generalized inverses, the essential components of such matrices, and associated matrix products. 
In addition, we derive more characterization conditions by investigating the positive integer powers of such matrices.
Collectively, these results extend the theory of orthogonal tripotent matrices in the dual quaternion setting and lay a solid foundation for their potential applications. 

\vspace{5pt}
The following topics are proposed for further research:
\begin{itemize}
    \item Further characterize dual orthogonal tripotent matrices by means of the rank and trace of dual quaternion matrices.
    \item Investigate applications of dual orthogonal tripotent matrices.
\end{itemize}

\backmatter


\section*{Declarations}
\textbf{Conflict of interest} \ \  The authors declare that they have no relevant financial or non-financial
competing interests.

\bibliography{sn-bibliography}

\begin{thebibliography}{20}
\providecommand{\natexlab}[1]{#1}
\providecommand{\url}[1]{{#1}}
\providecommand{\urlprefix}{URL }
\providecommand{\doi}[1]{\url{https://doi.org/#1}}
\providecommand{\eprint}[2][]{\url{#2}}
 \bibcommenthead

\bibitem[{Baksalary and Trenkler(2013)}]{Baksalary2}
Baksalary O, Trenkler G (2013) On $k${-}potent matrices. The Electronic Journal of Linear Algebra 26:446--470. \doi{10.13001/1081-3810.1664}

\bibitem[{Chen et~al(2024)Chen, Hu, Wang, and Guo}]{chen}
Chen S, Hu H, Wang S, et~al (2024) Dual quaternion matrices in precise formation flying of satellite clusters. Communications on Applied Mathematics and Computation pp 1--18. \doi{10.1007/s42967-024-00460-4}

\bibitem[{Clifford(1871)}]{Clifford}
Clifford (1871) Preliminary sketch of biquaternions. Proceedings of the London Mathematical Society 1(1):381--395. \doi{10.1112/PLMS/S1-4.1.381}

\bibitem[{Conway(2019)}]{conway}
Conway JB (2019) A course in functional analysis. Springer

\bibitem[{Cui and Qi(2025)}]{Cui1}
Cui C, Qi L (2025) A genuine extension of the {Moore--Penrose} inverse to dual matrices. Journal of Computational and Applied Mathematics 454:116185. \doi{10.1016/j.cam.2024.116185}

\bibitem[{Ding(2026)}]{Ding}
Ding W (2026) Algebraic method for eigenvalue problems of dual quaternion hermitian matrices and its application in rgb-hsv-based face representation and recognition. Applied Mathematics Letters 174:109829. \doi{10.1016/j.aml.2025.109829}

\bibitem[{Giribet et~al(2025)Giribet, Marciano, Mas, Ghersin, Villa, and Sarcinelli-Filho}]{Giribet}
Giribet JI, Marciano HN, Mas I, et~al (2025) Dual quaternion{-}based control for dynamic robot formations. In: 2025 International Conference on Unmanned Aircraft Systems, IEEE, pp 526--533, \doi{10.1109/ICUAS65942.2025.11007809}

\bibitem[{Gu and Luh(1987)}]{Gu}
Gu Y, Luh J (1987) Dual-number transformation and its applications to robotics. IEEE Journal on Robotics and Automation 3(6):615--623. \doi{10.1109/JRA.1987.1087138}

\bibitem[{Khatri(1980)}]{Khat}
Khatri C (1980) Powers of matrices and idempotency. Linear Algebra and its Applications 33:57--65. \doi{10.1016/0024-3795(80)90097-X}

\bibitem[{Li and Wang(2023)}]{Li1}
Li H, Wang H (2023) Weak dual generalized inverse of a dual matrix and its applications. Heliyon 9(6). \doi{10.1016/j.heliyon.2023.e16624}

\bibitem[{Mei et~al(2026{\natexlab{a}})Mei, Zuo, and Jiang}]{Mei2}
Mei T, Zuo K, Jiang W (2026{\natexlab{a}}) Orthogonal tripotent matrices. Linear and Multilinear Algebra pp 1--17. \doi{10.1080/03081087.2026.2689182}

\bibitem[{Mei et~al(2026{\natexlab{b}})Mei, Zuo, and Yan}]{mei1}
Mei T, Zuo K, Yan H (2026{\natexlab{b}}) Further results for the dual hartwig-spindelb{\" o}ck decomposition and its applications. Computational and Applied Mathematics 45(8):353. \doi{10.1007/s40314-026-03689-2}

\bibitem[{Pes and Rodriguez(2023)}]{pes}
Pes F, Rodriguez G (2023) A projection method for general form linear least-squares problems. Applied Mathematics Letters 145:108780. \doi{10.1016/j.aml.2023.108780}

\bibitem[{Qi and Luo(2023)}]{Qi1}
Qi L, Luo Z (2023) Eigenvalues and singular values of dual quaternion matrices. Pacific Journal of Optimization 19(2):257

\bibitem[{Qi et~al(2022)Qi, Ling, and Yan}]{Qi2}
Qi L, Ling C, Yan H (2022) Dual quaternions and dual quaternion vectors. Communications on Applied Mathematics and Computation 4(4):1494--1508. \doi{10.1007/s42967-022-00189-y}

\bibitem[{To{\v{s}}i{\'c} and Cvetkovi{\'c}-Ili{\'c}(2013)}]{To}
To{\v{s}}i{\'c} M, Cvetkovi{\'c}-Ili{\'c} DS (2013) The invertibility of the difference and the sum of commuting generalized and hypergeneralized projectors. Linear and Multilinear Algebra 61(4):482--493. \doi{10.1080/03081087.2012.689987}

\bibitem[{Udwadia et~al(2020)Udwadia, Pennestri, and de~Falco}]{Udwadia1}
Udwadia FE, Pennestri E, de~Falco D (2020) Do all dual matrices have dual {Moore--Penrose} generalized inverses? Mechanism and Machine Theory 151:103878. \doi{10.1016/j.mechmachtheory.2020.103878}

\bibitem[{Wang(2021)}]{Wang2}
Wang H (2021) Characterizations and properties of the {MPDGI} and {DMPGI}. Mechanism and Machine Theory 158:104212. \doi{10.1016/j.mechmachtheory.2020.104212}

\bibitem[{Xiao and Chen(2025)}]{xiao}
Xiao FZ, Chen LQ (2025) Unwinding{-}free property of the dual{-}quaternion{-}based pose tracking controllers designed by fully actuated system approaches. Aerospace Science and Technology 162:110197. \doi{10.1016/j.ast.2025.110197}

\bibitem[{Yang and Wang(2010)}]{Yang}
Yang J, Wang X (2010) The application of the dual number methods to scara kinematics. In: 2010 International Conference on Mechanic Automation and Control Engineering, IEEE, pp 3871--3874

\end{thebibliography}
\end{document}